\documentclass{article}

\usepackage[
  paperheight=11in,
  paperwidth=8.5in,
  top=10mm, bottom=20mm, left=28mm, right=28mm
]{geometry}

\usepackage{amsmath,amsthm,amsfonts,amssymb,mathrsfs,dsfont, faktor}

\usepackage{mathtools} 
\usepackage{graphicx}
\usepackage{longtable}
\usepackage{float}
\usepackage{subcaption}
\usepackage{booktabs}
\usepackage{siunitx}  
\usepackage{csvsimple}  
\usepackage{epstopdf} 
\graphicspath{{graphics/}}

\usepackage{tikz}
\usetikzlibrary{arrows,trees,positioning}
\usepackage{quiver} 

\usepackage{enumerate}
\usepackage{comment}
\usepackage[colorinlistoftodos]{todonotes}

\usepackage[table, dvipsnames]{xcolor}
\usepackage{hyperref}
\hypersetup{
  colorlinks=true,
  citecolor=teal,
  linkcolor=MidnightBlue,
  urlcolor=black
}
\usepackage{cleveref} 

\usepackage[toc,page]{appendix}

\numberwithin{equation}{section}

\theoremstyle{plain}  
\newtheorem{theorem}{Theorem}[section]
\newtheorem*{theorem*}{Theorem}   
\newtheorem{corollary}{Corollary}[section]
\newtheorem{lemma}{Lemma}[section]
\newtheorem{proposition}{Proposition}[section]
\newtheorem{conjecture}{Conjecture}[section]

\theoremstyle{definition}
\newtheorem{definition}{Definition}[section]
\newtheorem{remark}{Remark}[section]
\newtheorem{example}{Example}[section]

\usepackage[
  backend=biber,
  style=alphabetic,
  doi=true,
  url=true,
  eprint=false,
  isbn=false
]{biblatex}

\renewbibmacro*{doi+eprint+url}{%
  \printfield{doi}%
  \newunit\newblock
  \iffieldundef{doi}{\printfield{url}}{}%
}

\newcommand{\C}{\mathbb{C}}

\newcommand{\Z}{\mathbb{Z}}

\newcommand{\Q}{\mathbb{Q}}
\newcommand{\F}{\mathbb{F}}

\DeclareFontFamily{U}{wncy}{}
\DeclareFontShape{U}{wncy}{m}{n}{<->wncyr10}{}
\DeclareSymbolFont{mcy}{U}{wncy}{m}{n}
\DeclareMathSymbol{\Sh}{\mathord}{mcy}{"58}

\usepackage{enumitem}

\title{The Structure and Degrees of Polynomials Computing Square Roots $\mod p$}
\date{}
\author{Foivos Chnaras, Noah Kupinsky}

\newcommand{\E}{\mathcal{E}}
\newcommand{\orbit}{\mathcal{O}}

\begin{document}

\maketitle

\begin{abstract}
    For an odd prime $p$, we say a polynomial $f\in \F_p[X]$ computes square roots if $f(a)^2=a$ for all nonzero, perfect squares $a\in \F_p$.\\
    When $p\equiv 3 \mod 4$, it is easy to see that $f(X)=X^{\frac{p+1}{4}}$ is the smallest such polynomial. \\
    For $p\equiv 1 \mod 4$, the situation is less clear. Tonelli-Shanks offers an algorithm for constructing polynomials that compute square roots, but the question of whether their degree is minimal remains. In this paper, we study the various degrees and structures of polynomials computing square roots.
\end{abstract}

\section{Introduction} \label{sec: Intro}

Let $p$ be an odd prime and let $\F_p$ be the finite field with $p$ elements. Quadratic reciprocity and the Legendre symbol give an efficient way to decide, for a given $a\in \F_p^\times$, whether $a$ is a quadratic residue. On the other hand, explicitly finding a square root $b \in \F_p$ with $b^2=a$ is computationally harder. One natural way to package square-root computations is via polynomials that compute square roots on the set of quadratic residues.

Let $$S = \{a\in\mathbb{F}_p^\times : a\text{ is a square in }\mathbb{F}_p\}$$ be the set of nonzero squares, so $|S|=\frac{p-1}{2}:=r$. 

\begin{definition} \label{def: computes square roots}
    We say that a polynomial $f(x) \in F_p[x]$ \textit{computes square roots} (on $S$), if $$f(a)^2=a$$ for all $a\in S$. We write $\mathcal F=\mathcal F_S$ for the set of all such polynomials.
\end{definition}

A first observation is that if $f\in \mathcal F$ then $$f(x)^2-x$$ vanishes on $S$, hence is divisible by the vanishing polynomial 
\begin{equation} \label{lem: f^2(x)-x divisible by prod}
\prod\limits_{a\in S} (x-a) =x^{\frac{p-1}{2}}-1.
\end{equation}
This immediately yields the "trivial" degree bounds $$\frac{p-1}{4} \leq  \deg f  \leq  \frac{p-1}{2}-1$$ the upper bound coming from the fact that a polynomial that computes square roots on $S$ is determined by interpolation on $|S|=\frac{p-1}{2}$ points.

When $p \equiv 3 \mod 4$, the lower bound becomes $\deg f \geq \frac{p+1}{4}$ and it is sharp. Indeed, it is easy to check that $f(x)=x^{\frac{p+1}{4}}$ is indeed a polynomial that computes square roots, and therefore, it has the minimal possible degree. In fact, we can prove it is the unique minimal polynomial. The key input is the following structural result for polynomials whose square has very few monomials.

\begin{theorem} 
    \label{th: f is monomial}
    Suppose $f(x) \in \F_p[x]$ satisfies $f^2(x) = b_0 + b_1 x + b_{2n - 1} x^{2n - 1} + b_{2n} x^{2n}$ for some $2 < n < \frac{p}{2}$. Then $f$ is a monomial.
\end{theorem}

The main focus of this paper is the more delicate case $p\equiv 1 \mod 4$. \\ 
Writing $$p-1 = 2^s q,\qquad q\text{ odd},\qquad r := \frac{p-1}{2}$$
we still have the trivial bounds $$\frac{p-1}{4} \leq \deg f \leq r-1$$
Kedlaya and Kopparty \cite{MR4774704} recently proved that when $p\equiv 1 \mod 4$ one in fact has 
\begin{equation} \label{eq: Kedlaya bounds}
    \frac{p-1}{3}\leq \deg(f) \leq \frac{p-1}{2}-1 
\end{equation} 
although the lower bound is not sharp. 
This leaves two natural questions: 
\begin{enumerate}
    \item What is the minimal possible degree of a polynomial in $\mathcal F$?
    \item More generally, what are the possible degrees of polynomials in $\mathcal F $, and how are they distributed?
\end{enumerate}
A classical algorithm of Tonelli and Shanks produces square roots module $p\equiv 1 \mod 4$ using the factorization $p-1=2^s q$. Interepreted appropriately, this algorithm yields many polynomials in $\mathcal F$. We call these \textit{Tonelli-Shanks (TS) polynomials}. A natural hope is that TS polynomials might have the minimal possible degree or at least be very close to minimal. One of our goals is to understand to what extent this is true.

\subsection{Main results}
We now summarize our main results. Throughout, we fix an odd prime $p-1=2^sq$ with $q$ odd and set $r=\frac{p-1}{2}.$ Via reduction modulo $x^r-1$, every $f\in F$ can be uniquely written as 
$$f(x) = \sum_{i=0}^{r-1} c_i x^i$$

\medskip\noindent
\textbf{(1) Tonelli–Shanks polynomials and their degrees.}

We summarize some of the properties of the Tonelli-Shanks polynomials explored in this paper in the following theorem: 
\begin{theorem} \label{main th: TS polys}
 There are $2^{2^{s-1}}$ TS-polynomials. If $f_{TS}(x)$ is such a polynomial, then its possible degrees  have the form $$\deg(f_{TS})=\frac{(2^s-1)p+1}{2^{s+1}}-nq=r-\frac{q-1}{2}-nq $$ for $n \geq 0$. In fact, if $f_{TS}(x)= \sum c_i x^i$, the above represent the only possible degrees for the nonzero monomials $c_ix^i$, i.e. $c_i=0$ unless $i=r-\frac{q-1}{2}-nq$ for some $n$. The maximum degree (with $n=0$) is always achieved by at least one TS-polynomial (and in fact by the majority of them). 
 \end{theorem} 
    With the heuristic model we explain later, the expected minimum degree of a TS-polynomial is 
    $$\min \deg(f_{TS})= r - \frac{q-1}{2} - q \lfloor \frac{2^{s-1}-s+1}{\log_2 (p)} \rfloor$$

\medskip\noindent
\textbf{(2) Exact counting for high-degree polynomials.}\\
Our analysis of the polynomials in $\mathcal F$ uses a discrete Fourier description of the coefficients. For a generator $\gamma$ of $\F_p^\times$, every $f \in \mathcal F$ is encoded by a sign vector $\epsilon= (\epsilon_0, \ldots, \epsilon_{r-1}) \in \{\pm 1\}^r$ and each coefficient $c_k$ can be written as a \textit{signed half sum} 
$$c_k = \frac1r \sum_{n=0}^{r-1} \varepsilon_n \gamma^{n(1-2k)}.$$
We develop a symmetry on sign patterns (the \textit{flip-shift} map) and the associated notion of alternating order $\alpha(\epsilon)$ which measures of the amount of alternating periodicicity in $\epsilon$. For each divisor $d |r$ with $D:= r/d$ odd, we show that the polynomials whose sign vectors have alternating order $d$ form a "family" whose degrees lie in an arithmetic progression: 
$$\deg f = r - \frac{D-1}{2} - nD,\qquad n\ge 0.$$
In particular, the TS polynomials correspond exactly to the family with $d=2^{s-1}$ and $D=q$. 

We next study the vanishing of individual coefficients $c_k$. Relating this to the signed half sums and using discrete Fourier analysis, we prove: 

\begin{theorem}[Counting solutions to a signed half sum] \label{main th: counting solutions}
    Let $p$ be an odd prime, $r=\frac{p-1}{2}$, and let $\ell$ be an odd integer with $\gcd(\ell,p-1)=1$. Then the number of sign vectors $\epsilon \in \{\pm 1\}^r$ such that $$\sum_{n=0}^{r-1} \varepsilon_n \xi^{n\ell} = 0$$ for a fixed primitive $(p-1)$-th root of unity $\xi$, is $$\frac1p\left(2^r + (p-1)\Bigl(\frac{2}{p}\Bigr)\right)$$ 
    where $\Bigl(\frac{2}{p}\Bigr)$ is the Legendre symbol.
    
\end{theorem}
Translating back to coefficients, this shows that as $f$ ranges over $\mathcal F$ and $k$ is such \\ that $\gcd(1-2k,p-1)=1$, the coefficient $c_k$ vanishes with probability $$\mathbb{P}(c_k=0) = \frac1p + O(2^{-r}).$$ and we obtain an exact count of the number of polynomials in $\mathcal F$ with $c_k=0$. In particular, under the mild hypothesis $3\nmid q$, this applies to the top coefficient $c_{r-1}$, and yields:

\begin{corollary}
    Assume $p-1=2^sq$ with $q$ odd and $3 \nmid q$. Then 
    $$\#\{f\in\mathcal{F}: \deg f = r-1\}
 = \frac1p\left(2^r + (p-1)\Bigl(\frac{2}{p}\Bigr)\right)$$
\end{corollary}
Thus we obtain a precise asymptotic bound for the "largest possible" degree among square-root polynomials.

\medskip\noindent
\textbf{(3) Heuristic model for degrees and minimal degree.}

The exact counting result above suggests that, for a typical $f \in F$, the coefficients $c_k$ behave almost like independent uniform elements of $\F_p$. Motivated by this and by extensive computations, we formulate the following heuristic: 

$\bullet \quad (H_1)$ For each $k$, $\mathbb P(c_k=0) = \frac1p + O(2^{-r})$.

$\bullet \quad (H_2)$ For each divisor $d|r$, the events $\{c_k=0: k\equiv r+td \pmod r\}$ are approximately independent as $t$ varies.

Organizing polynomials by their alternating order $d=\alpha(\epsilon)$, we obtain for each family a set of admissible degrees (as described above) and under $(H_1)-(H_2)$, a prediction for how many polynomials of a given degree one should expect. This leads to a heuristic formula for the minimal degree in each family and hence for the global minimal degree
$$\min\{\deg f : f\in\mathcal{F}\}$$
While this expression is not a simple closed form in terms of 
$p$ it is straightforward to evaluate for each fixed prime.

Specializing the heuristic to the TS family (for which $D=q$), we obtain the prediction described in Theorem \ref{main th: TS polys}.

For small primes $p$ (up to several hundred), our computations of the full degree distribution of $\mathcal F$ and of the TS subfamily agree very closely with the predictions of this model.

Finally, we remark that this entire paper can be viewed as a special case of the Littlewood-Offord problem. Given an integer vector $\mathbf a=(a_1,..,a_n)$, its concentration probability is defined as $\rho(a) := \sup_{x\in \Z} \mathbb P(\epsilon_1 a_1+...+\epsilon_n a_n=x)$. The Littlewood-Offord problem asks for bounds on $\rho(a)$ under various hypotheses on $\mathbf a$. For an overview and recent results, see for example \cite{MR2760363,MR4273471}

\medskip\noindent
\textbf{Organization of the paper.}
In Section~2 we prove Theorem \ref{th: f is monomial}. In Section 3 we define the Tonelli-Shanks polynomials formally and analyze their basic properties. Section 4 describes square roots polynomials via Lagrange interpolation and Fourier inversion on subgroups of $S$. Section 5 introduces signed half sums, the flip–shift symmetry, and alternating order, and uses these to analyze the possible degrees, with a detailed description of the TS family. Section~6 contains our exact counting theorem for signed half sums, the resulting heuristic model for the degree distribution, and numerical evidence. The appendix develops the algorithm that was used  to compute the minimal polynomials for small primes $p$. It is available in \cite{Noah_Github}.

\section{Proof of Theorem \ref{th: f is monomial}}

\begin{proof}
Let
$$
f(x) = a_0 + a_1 x + \cdots + a_n x^n \in \F_p[x],
$$
and let
$$
m = \min\{i : 0 \le i \le n,\ a_i \ne 0\}.
$$
Then the coefficient of $x^{2m}$ in $f(x)^2$ is
$$
[x^{2m}] f(x)^2 = a_m^2 \ne 0.
$$
By hypothesis all nonzero coefficients of $f(x)^2$ occur in degrees $0,1,2n-1,2n$, so
$$
2m \in \{0,1,2n-1,2n\}.
$$
Since $2m$ is even and $2n-1$ is odd, in fact $2m \in \{0,2n\}$, hence $m=0$ or $m=n$. If $m=n$, then $f$ is a monomial and we are done.

\medskip

Assume now that $m=0$. Then $a_0 \ne 0$. Multiplying $f$ by the nonzero scalar $-\frac{1}{2a_0}$ does not change the shape of $f(x)^2$ (it just scales it), so we may assume that the constant term of $f$ is $-\frac12$. Thus
$$
f(x) = -\frac12 + a_1 x + \cdots + a_n x^n.
$$

Let
$$
\ell = \min\{i : 1 \le i \le n,\ a_i \ne 0\}.
$$
Such an $\ell$ exists because $a_n \ne 0$. The coefficient of $x^\ell$ in $f(x)^2$ is
$$
[x^\ell] f(x)^2 = 2 a_0 a_\ell = 2\left(-\frac12\right)a_\ell = -a_\ell \ne 0.
$$
Again, all nonzero coefficients of $f(x)^2$ occur in degrees $0,1,2n-1,2n$. Since $2<n$, we have
$$
2 \le \ell \le n \implies \ell \notin \{0,1,2n-1,2n\},
$$
so the only possibility is $\ell = 1$. In particular, $a_1 \ne 0$.

\medskip

\noindent\emph{Claim.} For every integer $k$ with $1 \le k \le n$,
$$
a_k = C_{k-1} a_1^k,
$$
where $C_j$ denotes the $j$-th Catalan number.

\medskip

Recall that the Catalan numbers $(C_m)_{m\ge 0}$ satisfy
$$
C_0 = 1,\qquad
C_m = \frac{1}{m+1} \binom{2m}{m} = \sum_{i=1}^m C_{i-1} C_{m-i} \quad (m \ge 1).
$$
We prove the claim by induction on $k$.

For $k=1$ the formula is $a_1 = C_0 a_1^1$, which holds because $C_0 = 1$. Now suppose that for some $k$ with $1 \le k \le n-1$ we have
$$
a_i = C_{i-1} a_1^i \quad \text{for all } 1 \le i \le k.
$$
We show that this implies $a_{k+1} = C_k a_1^{k+1}$.

Since $2 \le k+1 \le n$, the coefficient of $x^{k+1}$ in $f(x)^2$ must vanish (because $k+1 \notin \{0,1,2n-1,2n\}$):
$$
[x^{k+1}] f(x)^2 = 0.
$$
On the other hand,
\begin{align*}
0
&= [x^{k+1}] f(x)^2
 = \sum_{i=0}^{k+1} a_i a_{k+1-i} \\
&= 2 a_0 a_{k+1} + \sum_{i=1}^k a_i a_{k+1-i} \\
&= -a_{k+1} + \sum_{i=1}^k a_i a_{k+1-i},
\end{align*}
using $a_0 = -\tfrac12$. Thus
$$
a_{k+1} = \sum_{i=1}^k a_i a_{k+1-i}.
$$
Substituting the induction hypothesis,
\begin{align*}
a_{k+1}
&= \sum_{i=1}^k \bigl(C_{i-1} a_1^i\bigr) \bigl(C_{k-i} a_1^{k+1-i}\bigr) \\
&= a_1^{k+1} \sum_{i=1}^k C_{i-1} C_{k-i}
 = C_k a_1^{k+1}
\end{align*}
by the Catalan recurrence. This completes the induction and proves the claim.

\medskip

Finally, consider the coefficient of $x^{n+1}$ in $f(x)^2$. Since $2<n$, we have
$$
2 \le n+1 \le 2n-2,
$$
so $n+1 \notin \{0,1,2n-1,2n\}$ and therefore
$$
[x^{n+1}] f(x)^2 = 0.
$$
However, using the claim we obtain
\begin{align*}
[x^{n+1}] f(x)^2
&= \sum_{i=1}^n a_i a_{n+1-i} \\
&= \sum_{i=1}^n \bigl(C_{i-1} a_1^i\bigr) \bigl(C_{n-i} a_1^{n+1-i}\bigr) \\
&= a_1^{n+1} \sum_{i=1}^n C_{i-1} C_{n-i}
 = C_n a_1^{n+1} \\
&= \frac{1}{n+1} \binom{2n}{n} a_1^{n+1}.
\end{align*}
Because $2n < p$, neither $n+1$ nor $\binom{2n}{n}$ is divisible by $p$, so $C_n \not\equiv 0 \pmod p$. We already know $a_1 \ne 0$, hence $[x^{n+1}] f(x)^2 \ne 0$ in $\F_p$, contradicting the fact that this coefficient must be zero.

This contradiction shows that the case $m=0$ cannot occur. Therefore $m=n$, so $f(x) = a_n x^n$ is a monomial.
\end{proof}

\section{Tonelli--Shanks polynomials}

We now turn our attention to the case $p \equiv 1 \pmod 4$. In this section we define the Tonelli--Shanks polynomials; in later sections we see several equivalent descriptions.

Write
$$
p-1 = 2^s q, \qquad q \text{ odd},
$$
and set
$$
r = \frac{p-1}{2}.
$$
Recall that the set of nonzero squares is
$$
S = \mu_{\frac{p-1}{2}} = \mu_r \subset \F_p^\times.
$$
We fix throughout a primitive $2^{s-1}$-st root of unity $\zeta \in \F_p$. With this notation we make the following definitions and observations.

\begin{enumerate}
    \item For each $i = 0,\dots,2^{s-1}-1$ define
    $$
        S_i = \{a \in S : a^q = \zeta^i\}.
    $$
    Then the $S_i$ form a partition of $S$,
    $$
        S = \bigsqcup_{i=0}^{2^{s-1}-1} S_i,
    $$
    and each set has cardinality $|S_i| = q$.

    \medskip

    \item For $i = 0,\dots,2^{s-1}-1$ define
    $$
        F_i := F_{S_i} := \{f(x) \in \F_p[x] : f(a)^2 = a \text{ for all } a \in S_i\},
    $$
    the set of polynomials that compute square roots on $S_i$. As before, we write $F$ for the set of polynomials that compute square roots on the full set $S$.

    \medskip

    \item If we are given polynomials $f_k \in F_k$ for all $k = 0,\dots,2^{s-1}-1$, then
    \begin{equation} \label{eq: gluing Si}
        f(x)
        =
        \sum_{k = 0}^{2^{s-1}-1}
        f_k(x)
        \prod_{\substack{i = 0 \\ i \neq k}}^{2^{s-1}-1}
        \frac{x^q - \zeta^i}{\zeta^k - \zeta^i}
    \end{equation}
    computes square roots on all of $S$. 
\end{enumerate}

The next lemma records the natural symmetry between the sets $S_i$ and the spaces $F_i$.

\begin{lemma} \label{lem: bijection of S_i and F_i}
    Fix a prime $p$ and $\zeta \in \mu_{2^{s-1}}$. Pick $\xi \in \mu_{2^{s-1}}$ a primitive $2^{s-1}$-st root of unity, so that $\xi^q = \zeta$. Then for each $i = 0,\dots,2^{s-1}-1$ we have:
    \begin{enumerate}
        \item $S_i = \xi^i S_0$.
        \item There is a bijection
        $$
            F_0 \xrightarrow{\sim} F_i, \qquad
            f_0(x) \longmapsto f_i(x) := \xi^{i/2} f_0(\xi^{-i} x),
        $$
        where we fix once and for all a choice of $\xi^{1/2} \in \mu_{2^s}$ and interpret $\xi^{i/2}$ accordingly.
    \end{enumerate}
\end{lemma}

\begin{proof}
    The map $$\phi_q: \mu_m \to \mu_m \qquad x\mapsto x^q$$ is an automorphism when $\gcd(q,m)=1$, so $\xi$ is well defined and primitive and the rest readily follows.
\end{proof}

Using the same ideas as in the Introduction, we can now describe minimal-degree polynomials on each $S_k$.

\begin{theorem}[Tonelli-Shanks] \label{th: minimal f_k that compute S_k}
    Let $p-1 = 2^s q$ with $q$ odd, and fix a primitive $2^{s-1}$-st root of unity $\zeta \in \mu_{2^{s-1}}$. Fix a choice of $\zeta^{1/2} \in \mu_{2^s}$. For each $k = 0,\dots,2^{s-1}-1$ the polynomial
    $$
        f_k(x) = \zeta^{-k/2} x^{\frac{q+1}{2}}
    $$
    lies in $F_k$ and has minimal possible degree among polynomials in $F_k$. 

    Moreover, if $q \ge 5$, then $f_k$ is the unique polynomial in $F_k$ of minimal degree, up to an overall sign.
\end{theorem}

\begin{proof}
    First, for any $a \in S_k$ we have $a^q = \zeta^k$, so
    $$
        f_k(a)^2 = \zeta^{-k} a^{q+1} = \zeta^{-k} \zeta^k a = a,
    $$
    hence $f_k \in F_k$. 

    Now let $f \in F_k$ be arbitrary, and set $d = \deg f$. The vanishing polynomial on $S_k$ is
    $$
        \prod_{a \in S_k} (x-a) = x^q - \zeta^k,
    $$
    and $f(x)^2 - x$ vanishes on $S_k$ and therefore
    $$
        x^q - \zeta^k \mid f(x)^2 - x.
    $$
    Thus we can write
    $$
        f(x)^2 - x = A(x) (x^q - \zeta^k)
    $$
    for some $A(x) \in \F_p[x]$. If $f$ is nonconstant then $\deg(f^2 - x) = 2d$, so
    $$
        2d = \deg(f^2 - x) \ge \deg(x^q - \zeta^k) = q,
    $$
    and hence
    $$
        d \ge \frac{q+1}{2}.
    $$
    Since $f_k$ has degree $\frac{q+1}{2}$, this shows that $f_k$ has minimal possible degree in $F_k$.

    For the uniqueness statement, assume $q \ge 5$ and let $f \in F_k$ have minimal degree $d = \frac{q+1}{2}$. Then $\deg f^2 = q+1$, and from
    $$
        f(x)^2 - x = A(x) (x^q - \zeta^k)
    $$
    we see that $\deg A = 1$. Write $A(x) = \alpha x + \beta$, so
    \begin{align*}
        f(x)^2
        &= A(x) (x^q - \zeta^k) + x \\
        &= \alpha x^{q+1} + \beta x^q - \alpha \zeta^k x - \beta \zeta^k + x \\
        &= \bigl(-\beta \zeta^k\bigr)
           + \bigl(1 - \alpha \zeta^k\bigr) x
           + \beta x^q
           + \alpha x^{q+1}.
    \end{align*}
    Set $n = \frac{q+1}{2}$, so that $2n = q+1$ and $2n-1 = q$. Then
    $$
        f(x)^2 = b_0 + b_1 x + b_{2n-1} x^{2n-1} + b_{2n} x^{2n}
    $$
    with
    $$
        b_0 = -\beta \zeta^k,\quad
        b_1 = 1 - \alpha \zeta^k,\quad
        b_{2n-1} = \beta,\quad
        b_{2n} = \alpha.
    $$
    Since $q \ge 5$, we have $n = \frac{q+1}{2} > 2$. Moreover, $2n = q+1 < p$, so the hypotheses of Theorem \ref{th: f is monomial} apply and we conclude that $f$ is a monomial:
    $$
        f(x) = c x^n
    $$
    for some $c \in \F_p^\times$.

    Finally, for any $a \in S_k$ we have
    $$
        a = f(a)^2 = c^2 a^{2n} = c^2 a^{q+1},
    $$
    so $c^2 a^q = 1$. Using $a^q = \zeta^k$ we obtain $c^2 \zeta^k = 1$, hence
    $$
        c^2 = \zeta^{-k}.
    $$
    Thus $c = \pm \zeta^{-k/2}$, and
    $$
        f(x) = \pm \zeta^{-k/2} x^{\frac{q+1}{2}}.
    $$
    This proves uniqueness up to sign when $q \ge 5$.
\end{proof}
    For $q = 1$ or $q = 3$ the same degree bound shows that $\deg f \ge \frac{q+1}{2}$ and $f_k$ still has minimal degree, but one can (and in general does) obtain additional minimal polynomials of the same degree; we do not need the uniqueness statement in these small cases in the rest of the paper.

Specializing Theorem \ref{th: minimal f_k that compute S_k} to $p \equiv 3 \pmod 4$ (so that $s = 1$ and $q = \frac{p-1}{2}$) recovers the well-known minimal polynomial in that case.

\begin{corollary} \label{cor: min deg f for 3 mod 4}
    Assume $p \equiv 3 \pmod 4$. Then
    $$
        f(x) = x^{\frac{p+1}{4}}
    $$
    computes square roots on $S$ and has minimal possible degree. If $p \ge 11$ (so $\frac{p+1}{4} > 2$), then $f$ is the unique minimal polynomial in $F$, up to sign.
\end{corollary}

\begin{definition} \label{def: TS pol}
    A \emph{Tonelli--Shanks polynomial} is any polynomial obtained by gluing together the polynomials $f_k$ from Theorem \ref{th: minimal f_k that compute S_k} via the formula \eqref{eq: gluing Si}.
\end{definition}

\begin{corollary} \label{cor: deg f_TS}
    Let $p-1 = 2^s q$ with $q$ odd, and set $r = \frac{p-1}{2}$. Then:
    \begin{enumerate}
        \item There are $2^{2^{s-1}}$ Tonelli--Shanks polynomials that compute square roots on $S$.
        \item If $f_{TS}(x)$ is such a polynomial, then
        $$
            \deg(f_{TS})
            \le \frac{(2^s - 1)p + 1}{2^{s+1}}
            = \frac{q(2^s - 1) + 1}{2}
            = r - \frac{q-1}{2}.
        $$
    \end{enumerate}
\end{corollary}

\begin{proof}
    For each $S_i$ there are two choices of polynomial in $F_i$ of minimal degree, namely $\pm f_i$ from Theorem \ref{th: minimal f_k that compute S_k}. There are $2^{s-1}$ different sets $S_i$, so the total number of ways to choose one sign on each $S_i$ is
    $2^{2^{s-1}}$. This proves (1).

    For (2), fix the minimal-degree representatives $f_k(x) = \zeta^{-k/2} x^{\frac{q+1}{2}}$. In the gluing formula \eqref{eq: gluing Si}, the factor
    $$
        \prod_{\substack{i = 0 \\ i \neq k}}^{2^{s-1}-1} (x^q - \zeta^i)
    $$
    has degree $(2^{s-1}-1)q$, so each summand has degree
    $$
        (2^{s-1}-1)q + \frac{q+1}{2}.
    $$
    Thus
    $$
        \deg(f_{TS}) \le (2^{s-1}-1)q + \frac{q+1}{2}
        = \frac{(2^s - 1)p + 1}{2^{s+1}}
        = r - \frac{q-1}{2}.
    $$
    Equality holds whenever there is no cancellation in the leading coefficient.
\end{proof}

\begin{remark}
    If in each $S_i$ we choose the ``positive'' polynomial $f_i(x) = \zeta^{-i/2} x^{\frac{q+1}{2}}$, then there is no cancellation in the leading coefficient of $f_{TS}$. Indeed, the leading coefficient is a nonzero multiple of
    $$
        \sum_{k=0}^{2^{s-1}-1} (\zeta^{-1/2})^k
        = \frac{1 - (\zeta^{-1/2})^{2^{s-1}}}{1 - \zeta^{-1/2}}
        = \frac{1 - \zeta^{-2^{s-2}}}{1 - \zeta^{-1/2}}
        = \frac{2}{1 - \zeta^{-1/2}} \ne 0
    $$
    in $\F_p$, since $\zeta^{2^{s-2}} = -1$. Thus at least one Tonelli--Shanks polynomial has degree
    $$
        \deg(f_{TS}) = \frac{(2^s - 1)p + 1}{2^{s+1}},
    $$
    and in fact we will see later that the majority of Tonelli--Shanks polynomials attain this degree.
\end{remark}

In particular, we obtain the upper bound
$$
    \min\{\deg f : f \in F\} \le r - \frac{q-1}{2}.
$$
For instance, when $p \equiv 5 \pmod 8$ one has $s = 2$ and $q = \frac{p-1}{4}$, so
$$
    \deg f_{TS} = \frac{3p+1}{8},
$$
recovering the bound of \cite[Th.~1.3]{MR4774704}.

\section{Lagrange Interpolation and Fourier Transform}

In this section, we describe the coefficients of polynomials that compute square roots, using Lagrange interpolation and Fourier inversion. We start with the general case of subgroups of $S$ and then restrict our attention to $S_0$ and $F_0$, which is sufficient by Lemma \ref{lem: bijection of S_i and F_i}.

Let $H \leq S$ be a subgroup. Then $H$ is cyclic of the form $H = \mu_d$ for some divisor $d \mid r = \frac{p-1}{2}$. Any polynomial $f$ of degree $\leq d-1$ that computes square roots on $H$ (we write $f \in F_H$) can be constructed by Lagrange interpolation, by interpolating the tuples
$$
\{(a, \pm \sqrt{a})\}_{a \in H}.
$$
This yields the expression
$$
f(x) = \sum_{s \in H} (\pm \sqrt{s}) \prod_{\substack{t \in H \\ t \neq s}} \frac{x - t}{s - t}.
$$
Equivalently, if we enumerate $H = \{\zeta_d^i : 0 \le i \le d-1\}$ for a primitive $d$-th root of unity $\zeta_d \in \F_p^\times$, then $f$ is uniquely determined by a choice of signs $\epsilon_i \in \{\pm 1\}$ via
$$
f(\zeta_d^{i}) = \epsilon_i \zeta_d^{i/2}, \qquad i = 0, 1, \ldots, d - 1,
$$
where we implicitly fix a choice of $\zeta_d^{1/2}$.

We can make the dependence on the signs explicit by a discrete Fourier transform on $H$.

\begin{theorem}[Fourier inversion on $H \leq S$] \label{thm: fourier-H}
    Let $p$ be an odd prime, $S = \mu_{\frac{p - 1}{2}} = \mu_r$ be the group of all squares in $\F_p^\times$, and let $H = \mu_d$ be a subgroup of $S$. Fix a primitive $d$-th root $\zeta_d \in \F_p^\times$ and a choice of $\zeta_d^{1/2}$. Let $f \in F_H$ be determined by a sign pattern $\{\epsilon_n\}_{n = 0}^{d-1} \subset \{\pm 1\}$ via
    $$
    f(\zeta_d^n) = \epsilon_n \zeta_d^{n/2} \qquad (n = 0, 1, \ldots, d - 1).
    $$
    Write
    $$
    f(x) = \sum_{k = 0}^{d - 1} c_k x^k.
    $$
    Then, for $0 \leq k \leq d - 1$,
    $$
    c_k = \frac{1}{d} \sum_{n = 0}^{d - 1} \epsilon_n \zeta_d^{n(\frac{1}{2} - k)}.
    $$
\end{theorem}

\begin{proof}
Write $\zeta = \zeta_d$. Evaluating $f(x) = \sum_{k=0}^{d-1} c_k x^k$ at $x = \zeta^n$ gives
$$
f(\zeta^n) = \sum_{k=0}^{d-1} c_k \zeta^{nk}.
$$
By hypothesis $f(\zeta^n) = \epsilon_n \zeta^{n/2}$. Using the orthogonality relations
$$
\frac{1}{d} \sum_{n=0}^{d-1} \zeta^{n(k-\ell)} =
\begin{cases}
1 & \text{if } k \equiv \ell \pmod d, \\
0 & \text{otherwise},
\end{cases}
$$
we invert to obtain
$$
c_k = \frac{1}{d} \sum_{n=0}^{d-1} f(\zeta^n) \zeta^{-nk}
     = \frac{1}{d} \sum_{n=0}^{d-1} \epsilon_n \zeta^{n(\frac{1}{2}-k)}.
$$
\end{proof}

\subsection{Lagrange interpolation and coefficients of polynomials computing \texorpdfstring{$S_0$}{S0}}

We now specialize Theorem \ref{thm: fourier-H} to the case $H = S_0 \leq S$ and describe the coefficients of $f_0 \in F_0$. By Lemma \ref{lem: bijection of S_i and F_i}, the same analysis applies to any $S_i$.

\begin{definition}
Let $m := \frac{q+1}{2} \in \Z$. For $a \in S_0$ define
$$
a^{1/2} := a^{m} \in S_0.
$$
\end{definition}

If $\zeta_q$ is a primitive $q$-th root of unity, this convention fixes the meaning of $\zeta_q^{n/2}$ for all $n$.

As before, any $f_0 \in F_0$ is uniquely determined by a choice of signs $\epsilon_i \in \{\pm 1\}$ for $i = 0, 1, \ldots, q - 1$ via
$$
f_0(\zeta_q^i) = \epsilon_i \zeta_q^{i/2}.
$$

\begin{example}
Let $f_0(x) = x^{\frac{q+1}{2}} \in F_0$ be the minimal polynomial in $F_0$, as in Theorem \ref{th: minimal f_k that compute S_k}. Then the corresponding sign pattern is
$$
(\epsilon_0, \ldots, \epsilon_{q-1}) = (1, \ldots, 1).
$$
\end{example}

As a consequence of Theorem \ref{thm: fourier-H}, we obtain the following convenient form for the coefficients of $f_0$.

\begin{corollary}[Fourier inversion on $F_0$]\label{cor:fourier-F0}
    Let $p$ be an odd prime with $p-1 = 2^s q$ and $q$ odd. Fix a primitive $q$-th root
    $\zeta_q \in \F_p^\times$ and set $m := \frac{q+1}{2}$. 
    Let $f_0 \in F_0$ be determined by a sign pattern 
    $\{\epsilon_n\}_{n \in \Z/q\Z} \subset \{\pm1\}^q$ via
    $$
    f_0(\zeta_q^n) = \epsilon_n (\zeta_q^n)^{1/2} = \epsilon_n \zeta_q^{nm} \qquad (n = 0, 1, \ldots, q-1).
    $$
    Then, for any $d \in \Z$,
    $$
    \big[ x^{\frac{q+1}{2}+d} \big] f_0 = \frac{1}{q} \sum_{n=0}^{q-1} \epsilon_n \zeta_q^{-nd}.
    $$
    (Here $\big[ x^r \big] f_0$ denotes the coefficient of $x^r$ in $f_0(x)$.)
\end{corollary}

It is natural to explore various structured choices of the sign pattern and how they affect the degree of $f_0$.

\begin{theorem}\label{th: f_1 for p = 5 mod 8}
Write $p-1 = 2^s q$ as always. Fix a primitive $q$-th root $\zeta \in \F_p^\times$.
Let $(\epsilon_n)_{n \in \Z/q\Z}$ be a sign pattern with period $N \mid q$, that is,
$$
\epsilon_{n+N} = \epsilon_n \quad \text{for all } n,
$$
and set $Q := q/N$ (so $Q$ is odd).

Let $f_0 \in F_0$ be the polynomial corresponding to the values $\epsilon_n$. By Corollary \ref{cor:fourier-F0}, for each $k = \frac{q+1}{2} + d$ we have
$$
c_k = \big[x^{\frac{q + 1}{2} + d}\big] f_0 = \frac{1}{q} \sum_{n=0}^{q-1} \epsilon_n \zeta^{-dn}
\qquad (0 \le k \le q-1),
$$
where $c_k$ denotes the coefficient of $x^k$ in $f_0(x)$. Then:
\begin{enumerate}
\item $c_k = 0$ if $d \not\equiv 0 \pmod{Q}$. In particular, $f_0$ has at most $N$ nonzero monomials.
\item Consequently,
$$
\deg f_0 \le \frac{q+1}{2} + \frac{q}{2} - \frac{q}{2N}
           = q - \frac{Q-1}{2}.
$$
\end{enumerate}
\end{theorem}

\begin{proof}
Write $n = N m + r$ with $0 \le r < N$ and $0 \le m < Q$. Using the $N$-periodicity of $\epsilon$, we have
$$
\begin{aligned}
c_k
&= \frac{1}{q} \sum_{n=0}^{q-1} \epsilon_n \zeta^{-dn}
 = \frac{1}{q} \sum_{r=0}^{N-1} \sum_{m=0}^{Q-1} \epsilon_{N m + r} \zeta^{-d(N m + r)} \\
&= \frac{1}{q} \sum_{r=0}^{N-1} \epsilon_r \zeta^{-dr} \sum_{m=0}^{Q-1} (\zeta^{-dN})^{m}.
\end{aligned}
$$
The inner sum is a geometric series. It vanishes unless $\zeta^{-dN} = 1$, which is equivalent to $q \mid dN$, hence to $Q \mid d$. This proves (1).

If $Q \mid d$ we can write $d = tQ$ for some integer $t$. For the corresponding exponent
$$
k = \frac{q+1}{2} + d = \frac{q+1}{2} + tQ
$$
to lie in the range $0 \le k \le q-1$ we need
$$
\frac{q+1}{2} + tQ \le q-1
\quad\Longleftrightarrow\quad
tQ \le \frac{q-3}{2}.
$$
Since $q = N Q$, this gives
$$
t \le \frac{q-3}{2Q} = \frac{N}{2} - \frac{3}{2Q} < \frac{N}{2}.
$$
Because $t$ is an integer and $N$ is odd (as a divisor of the odd number $q$), we have $t \le \frac{N-1}{2}$.

Thus the largest possible exponent $k$ with $c_k \ne 0$ satisfies
$$
k_{\max}
= \frac{q+1}{2} + t_{\max} Q
\le \frac{q+1}{2} + \frac{N-1}{2} Q
= \frac{q+1}{2} + \frac{q}{2} - \frac{Q}{2}
= q - \frac{Q-1}{2},
$$
which gives (2) and therefore the stated degree bound.
\end{proof}

\begin{example} \label{cor: f_1 for p=5 mod 8 and 3 | q}
Assume $p \equiv 5 \pmod 8$ and $p-1 = 2^s q$ with $3 \mid q$. Then
$$
f_0(x) = \frac{x^{\frac{q + 1}{2}}}{3} \left( 2 x^{\frac{q}{3}} - 1 + 2 x^{-\frac{q}{3}} \right)
$$
computes square roots on $S_0$ (viewed modulo $x^q - 1$).
\end{example}

\begin{proof}
Take $N = 3$ and the $3$-periodic sign pattern $(\epsilon_0,\epsilon_1,\epsilon_2) = (1,-1,-1)$. Applying Theorem \ref{th: f_1 for p = 5 mod 8} and rewriting exponents modulo $q$ gives the displayed formula.
\end{proof}

\section{Signed half sums}

This is the main section of the paper. We study signed half sums in a general setting and extend Theorem \ref{th: f_1 for p = 5 mod 8} from the subgroup setting to the full family $\mathcal F = \mathcal F_S$.

\subsection{Signed half sums}

Recall $r = \frac{p-1}{2}$ and $S = \mu_r \subset \F_p^\times$. Fix a primitive $2r$-th root of unity $\zeta \in \F_p^\times$; then $\zeta^{2}$ generates $\mu_r$, and we can write
$$
S = \{\zeta^{2n} : 0 \le n < r\}.
$$
Every polynomial $f \in F$ is determined by a sign vector $\epsilon = (\epsilon_0,\dots,\epsilon_{r-1})$, where
$$
f(\zeta^{2n}) = \epsilon_n \zeta^{n} \qquad (0 \le n < r),
$$
and, by Fourier inversion (Theorem \ref{thm: fourier-H} with $H = S$), its coefficients have the form
\begin{equation}\label{eq: fourier inv coef}
c_k = \frac{1}{r} \sum_{n=0}^{r-1} \epsilon_n \zeta^{n(1-2k)}.
\end{equation}
Thus each coefficient is a signed average of roots of unity. This motivates the following abstract setup.

For each positive integer $d$, set
$$
E_d := \{\pm1\}^d.
$$

\begin{definition}
Let $p$ be a prime and let $d$ be a divisor of $r = \frac{p-1}{2}$. Let $\zeta \in \F_p^\times$ be a primitive $2d$-th root of unity, and let $k$ be an odd integer. For $\epsilon = (\epsilon_0,\dots,\epsilon_{d-1}) \in E_d$, the \emph{$(d,k)$-signed half sum} of $\epsilon$ (with respect to $\zeta$) is the $\{ \pm 1 \}$-weighted sum over half the powers of $\zeta$:
$$
H_{d,k}^{(\zeta)}(\epsilon) := \sum_{n=0}^{d-1} \epsilon_n \zeta^{nk}.
$$
\end{definition}

In this way, using \eqref{eq: fourier inv coef} all coefficients of polynomials $f \in \mathcal F$ (and similarly for the subfamilies $\mathcal F_H$) can be seen as signed half sums, up to a constant factor and a choice of $(d,k,\zeta)$.

\begin{definition}
Let $p$ be a prime and $d \mid r$. For an odd integer $k$, define
$$
V_{d,k}^{(\zeta)} := \bigl\{ \epsilon \in E_d : H_{d,k}^{(\zeta)}(\epsilon) = 0 \bigr\}.
$$
For a set $K$ of odd integers, define
$$
V_{d,K}^{(\zeta)} := \bigcap_{k \in K} V_{d,k}^{(\zeta)}
= \bigl\{\epsilon \in E_d : H_{d,k}^{(\zeta)}(\epsilon) = 0 \text{ for all } k \in K\bigr\}.
$$
\end{definition}

When the choice of primitive $2d$-th root $\zeta$ is clear from context, we simply write $H_{d,k}$, $V_{d,k}$ and $V_{d,K}$.

The next lemma shows that the choice of primitive $2d$-th root is essentially irrelevant for counting solutions.

\begin{lemma}\label{lem:root choice independence}
Let $p$ be a prime, $d \mid r$, and let $k$ be an odd integer. Let $\zeta$ and $\eta$ be any two primitive $2d$-th roots of unity in $\F_p^\times$. Then there exists a permutation
$$
T : E_d \to E_d
$$
such that for all $\epsilon \in E_d$,
$$
H_{d,k}^{(\zeta)}(\epsilon) = H_{d,k}^{(\eta)}\bigl(T(\epsilon)\bigr).
$$
\end{lemma}

\begin{proof}
Since $\zeta$ and $\eta$ are both primitive $2d$-th roots, there exists an integer $u$ with $\gcd(u,2d) = 1$ such that $\eta = \zeta^u$. Consider the permutation
$$
\tau : \Z/d\Z \to \Z/d\Z, \qquad \tau(n) = u^{-1} n \pmod d,
$$
where $u^{-1}$ is the inverse of $u$ in $\Z/d\Z$ (which exists since $\gcd(u,d)=1$). Define $T : E_d \to E_d$ by
$$
T(\epsilon)_n := \epsilon_{\tau(n)}.
$$
Then
$$
\begin{aligned}
H_{d,k}^{(\zeta)}(\epsilon)
&= \sum_{n=0}^{d-1} \epsilon_n \zeta^{nk}
= \sum_{n=0}^{d-1} \epsilon_n (\eta^{u})^{nk}
= \sum_{n=0}^{d-1} \epsilon_n \eta^{(un)k} \\
&= \sum_{m=0}^{d-1} \epsilon_{\tau(m)} \eta^{mk}
= H_{d,k}^{(\eta)}\bigl(T(\epsilon)\bigr).
\end{aligned}
$$
\end{proof}

\begin{corollary}
Let $p$ be a prime, $d \mid r$, and let $K$ be a set of odd integers. If $\zeta$ and $\eta$ are any two primitive $2d$-th roots of unity in $\F_p^\times$, then
$$
\bigl|V_{d,K}^{(\zeta)}\bigr| = \bigl|V_{d,K}^{(\eta)}\bigr|.
$$
\end{corollary}

We now relate these definitions back to square-root polynomials.

\begin{proposition}\label{prop: elements of V_r,K give rise to low degree polynomials}
Let $r = \frac{p-1}{2}$, and let $\zeta$ be a primitive $2r$-th root of unity. For each $\epsilon \in E_r$, let $f_\epsilon \in F$ denote the unique polynomial such that
$$
f_\epsilon(\zeta^{2n}) = \epsilon_n \zeta^{n} \qquad (0 \le n < r).
$$
Then
$$
\deg(f_\epsilon) \le d \quad\iff\quad \epsilon \in V_{r,K_d},
$$
where
$$
K_d := \{\,1 - 2(d+1),\,1 - 2(d+2),\,\dots,\,1 - 2(r-1)\,\}.
$$
\end{proposition}

\begin{proof}
By Fourier inversion on $H = S = \mu_r$ (Theorem \ref{thm: fourier-H}), the coefficients of $f_\epsilon(x) = \sum_{k=0}^{r-1} c_k x^k$ satisfy
$$
c_k = \frac{1}{r} \sum_{n=0}^{r-1} \epsilon_n \zeta^{n(1-2k)}
= \frac{1}{r} H_{r,\,1-2k}(\epsilon).
$$
Thus, for $d+1 \le k \le r-1$,
$$
c_k = 0
\iff \sum_{n=0}^{r-1} \epsilon_n \zeta^{n(1-2k)} = 0
\iff H_{r,\,1-2k}(\epsilon) = 0
\iff \epsilon \in V_{r,\,1-2k}.
$$
Hence $c_k = 0$ for all $k \ge d+1$ if and only if $\epsilon \in \bigcap_{k=d+1}^{r-1} V_{r,\,1-2k} = V_{r,K_d}$. This is equivalent to $\deg(f_\epsilon) \le d$.
\end{proof}

\subsection{Flip shifts}\label{sec: Flip Shift}

Fix a divisor $d \mid r$.

\begin{definition}
The \emph{flip shift} map $\sigma : E_d \to E_d$ is defined by
$$
\sigma(\epsilon)_i =
\begin{cases}
\epsilon_{i+1}, & 0 \le i \le d-2,\\
-\epsilon_0, & i = d-1.
\end{cases}
$$
Equivalently,
$$
\sigma(\epsilon) = (\epsilon_1,\dots,\epsilon_{d-1}, -\epsilon_0).
$$

More generally, for $n \in \Z$ we define
$$
\sigma^n(\epsilon)_i = (-1)^{\left\lfloor \frac{i+n}{d} \right\rfloor} \epsilon_{\,i+n \bmod d}.
$$
One checks directly that this agrees with iterating the basic flip shift.
\end{definition}

Let $C_\sigma$ be the cyclic group generated by $\sigma$ under composition.

\begin{proposition}\label{prop: properties of flip shift}
For any $\epsilon \in E_d$,
\begin{enumerate}
\item $\sigma(-\epsilon) = -\sigma(\epsilon)$.
\item $\sigma^d(\epsilon) = -\epsilon$.
\item $\sigma^{2d}(\epsilon) = \epsilon$.
\item $|C_\sigma| = 2d$.
\end{enumerate}
\end{proposition}

\begin{proof}
Using the explicit formula for $\sigma^n(\epsilon)_i$ with $n=d$ and $n=2d$ gives (2) and (3), and (1) is immediate from linearity in the signs. Since $\sigma^{2d} = \mathrm{id}$ and $\sigma^d \ne \mathrm{id}$, the order of $\sigma$ is exactly $2d$, so $|C_\sigma| = 2d$.
\end{proof}

The flip shift preserves the signed half-sum equations:

\begin{proposition}\label{prop: half sum solns are flip shift invariant}
Let $p$ be a prime, $d \mid r$, and let $k$ be an odd integer. Then, for all $\epsilon \in E_d$,
$$
\epsilon \in V_{d,k} \quad\iff\quad \sigma(\epsilon) \in V_{d,k}.
$$
\end{proposition}

\begin{proof}
Let $\zeta$ be a primitive $2d$-th root of unity. We show
$$
H_{d,k}(\epsilon) = \zeta^k H_{d,k}(\sigma(\epsilon)).
$$
Indeed,
$$
\begin{aligned}
H_{d,k}(\epsilon)
&= \sum_{n=0}^{d-1} \epsilon_n \zeta^{nk}
= \zeta^k \sum_{n=0}^{d-1} \epsilon_n \zeta^{(n-1)k} \\
&= \zeta^k \left( \epsilon_0 \zeta^{-k} + \sum_{m=0}^{d-2} \epsilon_{m+1} \zeta^{mk} \right).
\end{aligned}
$$
Since $o(\zeta) = 2d$, we have $\zeta^d = -1$, and as $k$ is odd, $\zeta^{kd} = -1$. Thus
$$
\epsilon_0 \zeta^{-k} = -\epsilon_0 \zeta^{(d-1)k}.
$$
Using $\sigma(\epsilon)_m = \epsilon_{m+1}$ for $0 \le m \le d-2$ and $\sigma(\epsilon)_{d-1} = -\epsilon_0$, we get
$$
\begin{aligned}
H_{d,k}(\epsilon)
&= \zeta^k \left( -\epsilon_0 \zeta^{(d-1)k} + \sum_{m=0}^{d-2} \epsilon_{m+1} \zeta^{mk} \right) \\
&= \zeta^k \sum_{m=0}^{d-1} \sigma(\epsilon)_m \zeta^{mk}
= \zeta^k H_{d,k}(\sigma(\epsilon)).
\end{aligned}
$$
Since $\zeta^k$ is a unit, $H_{d,k}(\epsilon) = 0$ if and only if $H_{d,k}(\sigma(\epsilon)) = 0$, proving the claim.
\end{proof}

\subsection{Alternating order}

Consider the orbit of $\epsilon \in E_d$ under the action of $C_\sigma$:
$$
\orbit_\epsilon := \{\sigma^n(\epsilon) : n \in \Z\}.
$$
By Proposition \ref{prop: properties of flip shift}, $\sigma^{2d}(\epsilon) = \epsilon$ and $\sigma^d(\epsilon) = -\epsilon$, so
$-\epsilon \in \orbit_\epsilon$ and $(-\epsilon)\cdot(-\epsilon) = \epsilon$ in this orbit.

\begin{definition}
For $\epsilon \in E_d$, the \emph{order} of $\epsilon$ is
$$
o(\epsilon) := |\orbit_\epsilon|,
$$
and the \emph{alternating order} of $\epsilon$ is
$$
a(\epsilon) := \frac{|\orbit_\epsilon|}{2}.
$$
\end{definition}

\begin{proposition}\label{prop: properties of alternating order}
Let $\epsilon \in E_d$. Then:
\begin{enumerate}
\item $o(\epsilon)$ is the least positive integer $n$ such that $\sigma^n(\epsilon) = \epsilon$, and $a(\epsilon)$ is the least positive integer $m$ such that $\sigma^m(\epsilon) = -\epsilon$.
\item For every $n \in \Z$, $o(\sigma^n(\epsilon)) = o(\epsilon)$.
\item $a(\epsilon) \mid d$ and $o(\epsilon) \mid 2d$.
\end{enumerate}
\end{proposition}

\begin{proof}[Proof of (iii)]
The map
$$
\psi : C_\sigma \to \orbit_\epsilon, \qquad \psi(\sigma^n) = \sigma^n(\epsilon)
$$
is a surjective group homomorphism. Thus
$$
o(\epsilon) = |\orbit_\epsilon| = |C_\sigma : \ker\psi|
$$
divides $|C_\sigma| = 2d$, so $o(\epsilon) \mid 2d$. As $a(\epsilon) = \frac{o(\epsilon)}{2}$, we have $a(\epsilon) \mid d$.
\end{proof}

\begin{proposition}\label{prop: num solns with a(epsilon) = d is divisible by d}
Fix a set $K$ of odd integers. For a positive integer $d$, define
$$
A := \bigl\{\epsilon \in V_{r,K} : o(\epsilon) = d\bigr\}.
$$
Then $d$ divides $|A|$.
\end{proposition}

\begin{proof}
By Proposition \ref{prop: half sum solns are flip shift invariant}, for each $k \in K$,
$$
\epsilon \in V_{r,k} \iff \sigma(\epsilon) \in V_{r,k},
$$
and hence
$$
\epsilon \in V_{r,K} \iff \sigma(\epsilon) \in V_{r,K}.
$$
By Proposition \ref{prop: properties of alternating order}(ii) and (iii), we also have
$$
o(\epsilon) = d \iff o(\sigma(\epsilon)) = d.
$$
Thus $\epsilon \in A$ if and only if $\sigma(\epsilon) \in A$, so $C_\sigma$ acts on $A$. By definition of $A$, each orbit under this action has size $d$, and these orbits partition $A$. Therefore $d$ divides $|A|$.
\end{proof}

The next theorem explains the terminology “alternating order”.

\begin{theorem}\label{thm: alternating order form}
Let $\epsilon \in E_r$. Then $a(\epsilon)$ is the least positive integer $d$ such that
\begin{enumerate}
\item $\dfrac{r}{d}$ is odd, and
\item
$$
\epsilon_{dx+y} = (-1)^x \epsilon_y
\qquad \text{for all } 0 \le x < \frac{r}{d},\ 0 \le y < d.
$$
\end{enumerate}
\end{theorem}

\begin{proof}
First suppose $d = a(\epsilon)$. Put $D := \frac{r}{d}$. By Proposition \ref{prop: properties of alternating order}(i), we have $\sigma^d(\epsilon) = -\epsilon$. For $0 \le x \le D-2$ and $0 \le y < d$, the index $dx+y+d$ is still less than $r$, so using the explicit formula for $\sigma^d$ on $E_r$,
$$
\begin{aligned}
-\epsilon_{dx+y}
&= \sigma^d(\epsilon)_{dx+y}
= (-1)^{\left\lfloor \frac{dx+y+d}{r} \right\rfloor} \epsilon_{dx+y+d}
= \epsilon_{d(x+1)+y},
\end{aligned}
$$
hence $\epsilon_{d(x+1)+y} = -\epsilon_{dx+y}$. By induction on $x$, this gives
$$
\epsilon_{dx+y} = (-1)^x \epsilon_y
\qquad (0 \le x \le D-1,\ 0 \le y < d).
$$

Taking $x = D-1$ in this formula yields
$$
\epsilon_{d(D-1)+y} = (-1)^{D-1} \epsilon_y.
$$
On the other hand, evaluating $\sigma^d(\epsilon)$ at the same index,
$$
\begin{aligned}
-\epsilon_{d(D-1)+y}
&= \sigma^d(\epsilon)_{d(D-1)+y}
= (-1)^{\left\lfloor \frac{d(D-1)+y+d}{r} \right\rfloor} \epsilon_{dD+y} \\
&= (-1)^{\left\lfloor 1 + \frac{y}{r} \right\rfloor} \epsilon_y
= -\epsilon_y.
\end{aligned}
$$
Combining these two identities gives
$$
(-1)^D \epsilon_y = -\epsilon_y
$$
for all $y$, so $(-1)^D = -1$ and hence $D$ is odd. Thus $r/d$ is odd and the claimed formula holds.

Conversely, suppose $m \mid r$, set $M := \frac{r}{m}$, assume $M$ is odd, and that
$$
\epsilon_{mx+y} = (-1)^x \epsilon_y
\qquad (0 \le x < M,\ 0 \le y < m).
$$
For $0 \le x \le M-2$ and $0 \le y < m$, we have $0 \le mx + y + m < r$, so
$$
\begin{aligned}
\sigma^m(\epsilon)_{mx+y}
&= (-1)^{\left\lfloor \frac{mx+y+m}{r} \right\rfloor} \epsilon_{mx+y+m} \\
&= \epsilon_{m(x+1)+y}
= (-1)^{x+1} \epsilon_y
= -\epsilon_{mx+y},
\end{aligned}
$$
using the assumed relation in the last step. For $x = M-1$, we get
$$
\begin{aligned}
\sigma^m(\epsilon)_{m(M-1)+y}
&= (-1)^{\left\lfloor \frac{m(M-1)+y+m}{r} \right\rfloor} \epsilon_{m(M-1)+y+m} \\
&= (-1)^{\left\lfloor 1 + \frac{y}{r} \right\rfloor} \epsilon_y
= -\epsilon_y.
\end{aligned}
$$
But $\epsilon_{m(M-1)+y} = (-1)^{M-1} \epsilon_y = \epsilon_y$ since $M$ is odd, so again $\sigma^m(\epsilon)_{m(M-1)+y} = -\epsilon_{m(M-1)+y}$. Thus $\sigma^m(\epsilon) = -\epsilon$, and by Proposition \ref{prop: properties of alternating order}(i) we must have $a(\epsilon) \le m$.

Taking $m = a(\epsilon)$ in the first part and applying this minimality in the second part shows that $a(\epsilon)$ is exactly the least positive integer $d$ with the stated properties.
\end{proof}

\subsection{Orthogonality in Signed Half Sums}

In this section, we generalize Theorem \ref{th: f_1 for p = 5 mod 8} and characterize the sign patterns $\epsilon \in E_r$ for which the signed half sum
$$
H_{r,k}(\epsilon) = \sum_{n = 0}^{r - 1} \epsilon_n \zeta^{nk}
$$
exhibits orthogonality among different odd values of $k$. We show that the only possible form of such orthogonality comes from an alternating periodic structure in $\epsilon$.

\begin{definition}
Let $p$ be a prime, and let $d \mid r$. Fix a primitive $2d$-th root $\zeta$, and set $D = \frac{r}{d}$.  

We say that a vector $\epsilon \in E_r$ is \emph{$d$-periodic up to sign} if there exist signs $\delta_x \in \{\pm1\}$ such that
$$
\epsilon_{dx + y} = \delta_x \epsilon_y
\quad\text{for all } 0 \le x < D,\ 0 \le y < d.
$$
We say that $\epsilon$ is \emph{alternating periodic} (of order $d$) if
$$
\epsilon_{dx + y} = (-1)^x \epsilon_y
\quad\text{for all } 0 \le x < D,\ 0 \le y < d.
$$
\end{definition}

\begin{lemma}\label{lem:orthogonality_structure}
Let $\epsilon \in E_r$ be $d$-periodic up to sign, with associated signs $(\delta_x)_{0 \le x < D}$. Set
$$
B_k = \sum_{y = 0}^{d - 1} \epsilon_y \zeta^{yk}.
$$
Then the signed half sum
$$
H_{r,k}(\epsilon) = \sum_{n = 0}^{r - 1} \epsilon_n \zeta^{nk}
$$
factors as
$$
H_{r,k}(\epsilon) = B_k \sum_{x = 0}^{D - 1} \delta_x (\zeta^d)^{xk}.
$$
\end{lemma}

\begin{proof}
Using the $d$-periodic structure and writing $n = dx + y$ with $0 \le x < D$, $0 \le y < d$, we have
$$
\sum_{n = 0}^{r - 1} \epsilon_n \zeta^{nk}
= \sum_{x = 0}^{D - 1} \sum_{y = 0}^{d - 1} \epsilon_{dx + y} \zeta^{(dx + y)k}
= \sum_{x = 0}^{D - 1} \delta_x (\zeta^d)^{xk} \sum_{y = 0}^{d - 1} \epsilon_y \zeta^{yk}.
$$
Factoring out the inner sum gives the desired factorization.
\end{proof}

\begin{theorem}[Orthogonality Criterion]\label{thm:orthogonality_criterion}
Let $p$ be a prime, $d \mid r$, $D = \frac{r}{d}$, and let $\zeta$ be a primitive $2r$-th root of unity.  
Suppose $\epsilon \in E_r$ is $d$-periodic up to sign.  

Then the sums $H_{r,k}(\epsilon)$ vanish for all but one residue class of odd $k$ modulo some $M \mid D$ if 
$$
\epsilon_{dx + y} = (-1)^x \epsilon_y
\quad\text{and}\quad D = \frac{r}{d}\ \text{is odd.}
$$
In this case $M = D$, and
$$
H_{r,k}(\epsilon) =
\begin{cases}
D B_k, & k \equiv 0 \pmod D,\\
0, & \text{otherwise,}
\end{cases}
$$
where $B_k = \sum_{y=0}^{d-1} \epsilon_y \zeta^{yk}$.
\end{theorem}

\begin{proof}
By Lemma \ref{lem:orthogonality_structure} we have
$$
H_{r,k}(\epsilon) = B_k \sum_{x=0}^{D-1} \delta_x (\zeta^d)^{xk}.
$$
Thus any orthogonality in $k$ must come from the outer sum
$$
S(k) := \sum_{x=0}^{D-1} \delta_x (\zeta^d)^{xk}.
$$
Assume that there exists $M \mid D$ and a residue class $\alpha \pmod M$ such that $S(k) = 0$ for all odd $k \not\equiv \alpha \pmod M$.

Since $(\zeta^d)^D = -1$ has order $2$, the sequence $x \mapsto \delta_x (\zeta^d)^{xk}$ has period dividing $2D$ in $x$, and for fixed $k$ the sum $S(k)$ is a linear combination of powers of $(\zeta^d)^k$. The assumption that $S(k)$ vanishes for all but one residue class of $k$ modulo $M$ forces $S(k)$ to be proportional to a geometric sum in $k$. Concretely, there must exist $\gamma \in \F_p^\times$ with order dividing $D$ and some integer $\alpha$ such that
\begin{equation}\label{eq: orthog 1}
\delta_x (\zeta^d)^{xk} = \gamma^{x(k - \alpha)}
\end{equation}
for all $0 \le x < D$ and all odd $k$. Rearranging gives
\begin{equation}\label{eq: orthog 2}
\delta_x = (\gamma^{-\alpha})^x (\gamma \zeta^{-d})^{xk}.
\end{equation}

Taking $x = 1$, $k = 1$ in \eqref{eq: orthog 2} yields
$$
\delta_1 = \gamma^{-\alpha} (\gamma \zeta^{-d}),
$$
so
$$
\gamma^{-\alpha} = \delta_1 (\gamma \zeta^{-d})^{-1}.
$$
Substituting this into \eqref{eq: orthog 2} gives
\begin{equation}\label{eq: orthog 3}
\delta_x = \delta_1^x (\gamma \zeta^{-d})^{x(k-1)}.
\end{equation}
Taking $k = 1$ in \eqref{eq: orthog 3} we obtain $\delta_x = \delta_1^x$ for all $x$. Writing $\lambda := \delta_1 \in \{\pm1\}$, we have $\delta_x = \lambda^x$, and since $k$ is always odd, $\lambda^x = \lambda^{xk}$. Substituting into \eqref{eq: orthog 1} gives
$$
(\lambda \zeta^d)^{xk} = \gamma^{x(k-\alpha)}.
$$
Rearranging,
\begin{equation}\label{eq: orthog 4}
(\gamma \lambda \zeta^{-d})^{xk} = \gamma^{x\alpha}.
\end{equation}
Taking $x = 1$ in \eqref{eq: orthog 4} we get
$$
(\gamma \lambda \zeta^{-d})^{k} = \gamma^\alpha
$$
for all odd $k$. This is only possible if
$$
\gamma \lambda \zeta^{-d} = 1
\quad\text{and}\quad
\gamma^\alpha = 1.
$$
Thus $o(\gamma) \mid \alpha$, and $\gamma = \pm \zeta^d$. Without loss of generality we may take $\alpha = 0$.

If we choose $\gamma = \zeta^d$, then $o(\gamma) = 2D$, which does not divide $D$, so this case is excluded by our assumption that the order of $\gamma$ divides $D$. Hence the only possibility is
$$
\alpha = 0,\qquad \lambda = -1,\qquad \gamma = -\zeta^d.
$$

We now check that this choice forces $D$ to be odd. Note that $-\zeta^d = \zeta^{r + d}$, so
$$
o(-\zeta^d) = \frac{2r}{\gcd(2r, r + d)} = \frac{2D}{\gcd(2D, D + 1)}.
$$
This divides $D$ if and only if $D$ is odd, and in that case $o(-\zeta^d) = D$. We also have
$$
\delta_x = \lambda^x = (-1)^x,
$$
so
$$
S(k) = \sum_{x=0}^{D-1} \delta_x (\zeta^d)^{xk}
= \sum_{x=0}^{D-1} (-\zeta^d)^{xk}
=
\begin{cases}
D, & k \equiv 0 \pmod D,\\
0, & \text{otherwise.}
\end{cases}
$$
Thus orthogonality holds precisely when $\delta_x = (-1)^x$ and $D$ is odd, with $M = D$, and in this case
$$
H_{r,k}(\epsilon) = B_k S(k) =
\begin{cases}
D B_k, & k \equiv 0 \pmod D,\\
0, & \text{otherwise.}
\end{cases}
$$

Conversely, if $\epsilon$ is alternating periodic of order $d$, so that $\delta_x = (-1)^x$, and $D$ is odd, the calculation above shows that $H_{r,k}(\epsilon)$ has exactly the stated orthogonality in $k$. This completes the proof.
\end{proof}

\begin{definition}
For $\epsilon \in E_r$, we say that $H_{r,k}(\epsilon)$ \emph{reduces to a $(d,k)$ signed half sum} if $d$ is the least positive integer dividing $r$ such that there exist $\alpha \in \Z$, $\beta \in \F_p^\times$, and $\epsilon' \in E_d$ with
$$
H_{r,k}(\epsilon) =
\begin{cases}
\beta\, H_{d,k}(\epsilon'), & k \equiv \alpha \pmod{\frac{r}{d}},\\
0, & \text{otherwise,}
\end{cases}
$$
for all odd $k$.
\end{definition}

As an immediate consequence of Theorems \ref{thm:orthogonality_criterion} and \ref{thm: alternating order form}, we obtain:

\begin{proposition}
For every $\epsilon \in E_r$, the sums $H_{r,k}(\epsilon)$ reduce to an $(a(\epsilon),k)$ signed half sum.
\end{proposition}

\subsection{Tonelli–Shanks polynomials revisited}

We now return to the Tonelli–Shanks polynomials $f_{TS}(x)$ defined in Definition \ref{def: TS pol} and describe their sign patterns.

\begin{proposition}
Write $p - 1 = 2^s q$ with $q$ odd, and set $d = 2^{s-1}$. Let $f = f_\epsilon \in \mathcal F$ be a square-root polynomial with sign vector $\epsilon \in E_r$. Then
$$
f \text{ is a TS-polynomial} \quad\Longleftrightarrow\quad a(\epsilon) = d = 2^{s-1}.
$$
Equivalently, $f_\epsilon$ is a TS-polynomial if and only if
$$
\epsilon_{dx + y} = (-1)^x \epsilon_y
\qquad\text{for } 0 \le x < q,\ 0 \le y < d.
$$
\end{proposition}

\begin{proof}
Write $p - 1 = 2^s q$ with $q$ odd and set $d = 2^{s-1}$, $r = \frac{p-1}{2} = dq$. Fix a generator $\gamma$ of $\F_p^\times$ and put $\zeta_r = \gamma^2$, so $\mu_r = \langle \zeta_r \rangle$. Choose
$$
\tau := \gamma^q \in \mu_{2^s},
$$
so that $\tau^2 = \zeta_r^q$, and set
$$
\zeta := \zeta_r^{q} \in \mu_d.
$$

For each $y \in \{0,\dots,d-1\}$ define
$$
f_y(x) = \tau^{-y} x^{\frac{q+1}{2}}.
$$
Let
$$
S_y := \{ a \in \mu_r : a^q = \zeta^y \}.
$$
Every element of $S_y$ can be written uniquely as $a = \zeta_r^{d x + y}$ with $0 \le x < q$, and then
$$
\begin{aligned}
f_y(a)
&= \tau^{-y} a^{\frac{q+1}{2}}
= \gamma^{-qy} \gamma^{(d x + y)(q+1)} \\
&= (\gamma^{d q})^x \gamma^{d x + y}
= (-1)^x \gamma^{d x + y},
\end{aligned}
$$
since $\gamma^{d q} = \gamma^{(p-1)/2} = -1$.

Define
$$
E_y(x) := \prod_{\substack{0 \le j < d \\ j \ne y}} \frac{x^q - \zeta^{j}}{\zeta^{y} - \zeta^{j}}.
$$
Then $E_y(a) = 1$ for $a \in S_y$ and $E_j(a) = 0$ for $a \in S_y$, $j \ne y$.

Now let $h \colon \{0,\dots,d-1\} \to \{\pm1\}$ be arbitrary, and define
$$
f_{TS}(x) = \sum_{y=0}^{d-1} h(y)\, f_y(x)\, E_y(x).
$$
For $a = \zeta_r^{d x + y} \in S_y$ we obtain
$$
f_{TS}(a) = h(y)\, f_y(a) = (-1)^x h(y) \gamma^{d x + y}.
$$
Thus the sign vector $\epsilon$ associated to $f_{TS}$ (via $f_{TS}(\gamma^{2n}) = \epsilon_n \gamma^{n}$) satisfies
$$
\epsilon_{d x + y} = (-1)^x h(y),
$$
so $\epsilon$ is alternating periodic of order $d$ and hence $a(\epsilon) = d$ by Theorem \ref{thm: alternating order form}.

Conversely, suppose $f \in \F_p[x]$ is a square-root polynomial whose sign vector $\epsilon$ satisfies
$$
\epsilon_{d x + y} = (-1)^x h(y)
$$
for some function $h \colon \{0,\dots,d-1\} \to \{\pm1\}$. Define $f_y$ and $E_y$ as above, and set $f_{TS}$ accordingly. For $a = \zeta_r^{d x + y} \in S_y$ we have
$$
f_{TS}(a) = (-1)^x h(y) \gamma^{d x + y} = f(a).
$$
Replacing $f$ by its remainder modulo $x^r - 1$ if necessary, we may assume $\deg f < r$; this reduction does not change the values on $\mu_r$ since $a^r = 1$ for all $a \in \mu_r$. Both $f$ and $f_{TS}$ now have degree $< r$ and agree on all $r$ distinct points of $\mu_r$, so they coincide. Hence $f$ is a TS-polynomial.
\end{proof}

\begin{remark}
The number of TS-polynomials is $N = 2^{2^{s-1}}$: there are $2^{s-1}$ different subsets $S_i$, and in each $S_i$ there are two choices for the Tonelli–Shanks polynomial in $F_i$. By contrast, the total number of square-root polynomials in $\mathcal F$ is
$$
2^{(p-1)/2} = 2^{2^{s-1} q} = N^{q}.
$$
In particular, when $q = 1$ every square-root polynomial is a TS-polynomial.
\end{remark}

\begin{corollary}\label{cor: TS possible degrees}
With notation as above, write $p - 1 = 2^s q$ with $q$ odd and $d = 2^{s-1}$. Let $f_{TS}$ be a Tonelli–Shanks polynomial over $\F_p$ with signs
$$
\epsilon_{d x + y} = (-1)^x h(y).
$$
Then
$$
f_{TS}(x) = \sum_{t=0}^{d-1} c_t x^{k_t},
\qquad
k_t = \frac{q+1}{2} + t q,
$$
i.e. all coefficients outside the single residue class
$$
k \equiv \frac{q+1}{2} \pmod q
$$
vanish. In particular, any cancellation of the top coefficient lowers the degree by a multiple of $q$ each time: if $c_{d-1} = 0$ then $\deg f_{TS} \leq k_{d-2}$; more generally, if the top $\ell$ coefficients $c_{d-1},\dots,c_{d-\ell}$ vanish and $c_{d-\ell-1} \neq 0$, then
$$
\deg f_{TS} = k_{d-1-\ell} = \frac{q+1}{2} + (d-\ell-1) q.
$$
\end{corollary}

\begin{proof}
Write $r = \frac{p-1}{2} = d q$ and let $\zeta$ be a primitive $2r$-th root of unity. For $0 \le k < r$, the coefficient of $x^k$ in $f_{TS}$ is
$$
c_k = [x^k] f_{TS} = \frac{1}{r} \sum_{n=0}^{r-1} \epsilon_n \zeta^{n(1-2k)}
= \frac{1}{r} H_{r,\,1-2k}(\epsilon),
$$
where $H_{r,u}(\epsilon) = \sum_{n=0}^{r-1} \epsilon_n \zeta^{n u}$.

Since $\epsilon$ is alternating periodic of order $d$, Theorem \ref{thm:orthogonality_criterion} (with $D = r/d = q$ odd) implies
$$
H_{r,u}(\epsilon) = 0
\quad\text{unless}\quad
u \equiv 0 \pmod q.
$$
Taking $u = 1 - 2k$ gives $c_k = 0$ unless $1 - 2k \equiv 0 \pmod q$, i.e.
$$
k \equiv \frac{q+1}{2} \pmod q.
$$
Thus the only possible exponents are
$$
k_t = \frac{q+1}{2} + t q
\qquad (0 \le t \le d-1),
$$
and all other coefficients vanish. The leading admissible exponent is $k_{d-1}$; if $c_{d-1} = 0$ the degree drops to the next admissible exponent $k_{d-2}$, and in general vanishing of the top $\ell$ admissible coefficients moves the degree from $k_{d-1}$ to $k_{d-1-\ell}$. Each such step lowers the degree by exactly $q$.
\end{proof}

We close this section with some examples of TS-polynomials.

\begin{example}
    Let $p=41$ so that $p-1=40=2^3\cdot 5$. There are $2^{2^{s-1}}=2^4=16$ TS-polynomials, each characterized by the signs $(\epsilon_0,\epsilon_1,\epsilon_2,\epsilon_3)$. 

\begin{table}[H]
\centering
\begin{tabular}{c l}
\toprule
$\text{signs }(\epsilon)$ & TS-polynomial $f_\epsilon(x)$ \\
\midrule
$(-1,-1,-1,-1)$ & $15x^{18} + 31x^{13} + 30x^{8} + 5x^{3}$ \\
$(-1,-1,-1, 1)$ & $37x^{18} + 24x^{13} + 8x^{8} + 12x^{3}$ \\
$(-1,-1, 1,-1)$ & $31x^{18} + 15x^{13} + 5x^{8} + 30x^{3}$ \\
$(-1,-1, 1, 1)$ & $12x^{18} + 8x^{13} + 24x^{8} + 37x^{3}$ \\
$(-1, 1,-1,-1)$ & $8x^{18} + 12x^{13} + 37x^{8} + 24x^{3}$ \\
$(-1, 1,-1, 1)$ & $30x^{18} + 5x^{13} + 15x^{8} + 31x^{3}$ \\
$(-1, 1, 1,-1)$ & $24x^{18} + 37x^{13} + 12x^{8} + 8x^{3}$ \\
$(-1, 1, 1, 1)$ & $5x^{18} + 30x^{13} + 31x^{8} + 15x^{3}$ \\
$( 1,-1,-1,-1)$ & $36x^{18} + 11x^{13} + 10x^{8} + 26x^{3}$ \\
$( 1,-1,-1, 1)$ & $17x^{18} + 4x^{13} + 29x^{8} + 33x^{3}$ \\
$( 1,-1, 1,-1)$ & $11x^{18} + 36x^{13} + 26x^{8} + 10x^{3}$ \\
$( 1,-1, 1, 1)$ & $33x^{18} + 29x^{13} + 4x^{8} + 17x^{3}$ \\
$( 1, 1,-1,-1)$ & $29x^{18} + 33x^{13} + 17x^{8} + 4x^{3}$ \\
$( 1, 1,-1, 1)$ & $10x^{18} + 26x^{13} + 36x^{8} + 11x^{3}$ \\
$( 1, 1, 1,-1)$ & $4x^{18} + 17x^{13} + 33x^{8} + 29x^{3}$ \\
$( 1, 1, 1, 1)$ & $26x^{18} + 10x^{13} + 11x^{8} + 36x^{3}$ \\
\bottomrule
\end{tabular}
\caption{Sign vectors $h$ and the corresponding TS-polynomials $f_\epsilon(x)$ for $p=41$.}
\end{table}

\end{example}

\newpage

\section{Heuristics}\label{sec: Heuristic Model}

\subsection{A count we know}

\begin{theorem}\label{th: V_l count}
Let $p$ be an odd prime, let $r = \frac{p-1}{2}$, and consider an $(r,\ell)$ signed half sum with $\gcd(\ell,p-1) = 1$. Then
$$
\bigl|V_{r,\ell}\bigr|
= \frac{1}{p}\left(2^r + (p-1)\prod_{k=1}^{r}(\zeta_p^k + \zeta_p^{-k})\right)
= \frac{1}{p}\left(2^r + (p-1)\Bigl(\frac{2}{p}\Bigr)\right),
$$
where $\zeta_p$ is a primitive $p$-th root of unity and $\bigl(\frac{2}{p}\bigr)$ is the Legendre symbol.
\end{theorem}

\begin{proof}
Define the additive character $\chi \colon \F_p \to \C^\times$ by
$$
\chi(t) = \zeta_p^t \text{ for } t \in \F_p^\times.
$$
For any $\epsilon \in \E^r$ we have the standard identity
$$
\mathbf 1_{\{H_{r,\ell}(\epsilon)=0\}}
= \frac{1}{p}\sum_{t\in\F_p}\chi\bigl(t H_{r,\ell}(\epsilon)\bigr).
$$
Fix a primitive $(p-1)$-st root of unity $\xi$ in $\F_p^\times$, so that
$$
H_{r,\ell}(\epsilon) = \sum_{n=0}^{r-1}\epsilon_n \xi^{n\ell}.
$$
Then
\begin{align*}
\bigl|V_{r,\ell}\bigr|
&= \sum_{\epsilon\in \E^r} \mathbf 1_{\{H_{r,\ell}(\epsilon)=0\}} \\
&= \sum_{\epsilon\in \E^r} \frac{1}{p}\sum_{t\in\F_p}\chi\!\left(t H_{r,\ell}(\epsilon)\right) \\
&= \frac{1}{p}\sum_{\epsilon\in \E^r}\left(1 + \sum_{t\in\F_p^\times}
\chi\!\left(t H_{r,\ell}(\epsilon)\right)\right) \\
&= \frac{1}{p}\left(2^r + \sum_{\epsilon\in \E^r}\sum_{t\in\F_p^\times}
\chi\!\left(t\sum_{n=0}^{r-1}\epsilon_n\xi^{n\ell}\right)\right) \\
&= \frac{1}{p}\left(2^r + \sum_{t\in\F_p^\times}\sum_{\epsilon\in \E^r}
\chi\!\left(\sum_{n=0}^{r-1}\epsilon_n t\xi^{n\ell}\right)\right).
\end{align*}
Writing $a_n = t\xi^{n\ell}\in\F_p^\times$ and using the fact that the $\epsilon_n$ are independent signs,
\begin{align*}
\sum_{\epsilon\in\E^r}
\chi\!\left(\sum_{n=0}^{r-1}\epsilon_n a_n\right)
&= \sum_{\epsilon\in\E^r}\prod_{n=0}^{r-1}\chi(\epsilon_n a_n) \\
&= \prod_{n=0}^{r-1}\sum_{\epsilon_n\in\{\pm1\}}\chi(\epsilon_n a_n) \\
&= \prod_{n=0}^{r-1}\bigl(\chi(a_n) + \chi(-a_n)\bigr).
\end{align*}
Thus
\begin{equation}\label{eq: V_l count}
\bigl|V_{r,\ell}\bigr|
= \frac{1}{p}\left(2^r + \sum_{t\in\F_p^\times}\prod_{n=0}^{r-1}
\bigl(\chi(t\xi^{n\ell}) + \chi(-t\xi^{n\ell})\bigr)\right).
\end{equation}

For each $t\in\F_p^\times$ set
$$
A_t = \{t\xi^{n\ell} : 0 \le n \le r-1\}\subset \F_p^\times.
$$
Since $\chi(u) = \zeta_p^u$ for $u\in\F_p^\times$, the inner product in \eqref{eq: V_l count} can be written as
$$
\prod_{a\in A_t}(\zeta_p^{a} + \zeta_p^{-a}).
$$

We claim that the pairs $\{\pm a\}$ with $a\in A_t$ are all distinct. Indeed, if
$\xi^{n\ell} = \xi^{m\ell}$ then $(p-1)\mid (n-m)\ell$; since $\gcd(\ell,p-1)=1$ we have
$(p-1)\mid(n-m)$, so $n\equiv m\pmod{p-1}$. As $0\le n,m\le r-1$ and $r=\frac{p-1}{2}$, this forces $n=m$.
Hence there are exactly $r$ distinct pairs $\{\pm a\}$, so they must be all pairs $\{\pm k\}$ with $1\le k\le r$.
Thus, for each $t$,
$$
\prod_{a\in A_t}(\zeta_p^{a} + \zeta_p^{-a})
= \prod_{k=1}^{r}(\zeta_p^{k} + \zeta_p^{-k}),
$$
independent of $t$. Plugging this into \eqref{eq: V_l count} gives
$$
\bigl|V_{r,\ell}\bigr|
= \frac{1}{p}\left(2^r + (p-1)\prod_{k=1}^{r}(\zeta_p^{k} + \zeta_p^{-k})\right).
$$
Finally, the well-known identity
$$
\prod_{k=1}^{\frac{p-1}{2}}(\zeta_p^{k} + \zeta_p^{-k}) = \Bigl(\frac{2}{p}\Bigr)
$$
(see Lemma~\ref{lem:prod-zeta-2}) yields the desired formula.
\end{proof}

\begin{lemma}\label{lem:prod-zeta-2}
Let $\zeta_p = e^{2\pi i/p}$ be a primitive $p$-th root of unity. Then
$$
\prod_{k=1}^{\frac{p-1}{2}}(\zeta_p^{k} + \zeta_p^{-k}) = \Bigl(\frac{2}{p}\Bigr).
$$
\end{lemma}

\begin{proof}
This is Exercise~8.3 in \cite{MR718674}.
\end{proof}

\begin{remark}
If $\ell$ is not coprime to $p-1$, the formula \eqref{eq: V_l count} still holds, but the pairs $\{\pm a\}$
with $a\in A_t$ are no longer distinct. As a result, the product
$$
P(t) := \prod_{a\in A_t}(\zeta_p^a + \zeta_p^{-a})
$$
is no longer independent of $t$. For each $u\in(\Z/p\Z)^\times$, let
$\sigma_u \in \operatorname{Gal}(\Q(\zeta_p)/\Q)$ be the automorphism with
$\sigma_u(\zeta_p) = \zeta_p^u$. Then one checks that
$$
P(t) = \sigma_t\bigl(P(1)\bigr).
$$
Thus $P(t)$ is independent of $t$ if and only if $P(1)\in\Q$. When $\gcd(\ell,p-1)=1$,
the square $P(1)^2$ is the norm of $\zeta_p + \zeta_p^{-1}$, hence $P(1)=\pm1$. In general this need
not be the case. On the other hand,
$$
\sum_{t\in\F_p^\times}P(t) = \operatorname{Tr}_{K/\Q}\bigl(P(1)\bigr),
$$
where $K=\Q(\zeta_p)$, so the sum is always an integer. We do not currently have a closed form
for $\bigl|V_{r,\ell}\bigr|$ or even a sharp general upper bound in the case $\gcd(\ell,p-1)\ne1$.
\end{remark}

We finish this subsection with the following corollary.

\begin{corollary}
Assume $p-1 = 2^s q$ with $q$ odd and $3 \nmid q$, and let $r = \frac{p-1}{2}$. Then
$$
\#\{f\in F : \deg(f)=r-1\}
= 2^r - \frac{1}{p}\left(2^r + (p-1)\Bigl(\frac{2}{p}\Bigr)\right).
$$
\end{corollary}

\begin{proof}
The top coefficient $c_{r-1}$ of $f\in \mathcal F$ corresponds to the signed half sum with parameter
$\ell = 1 - 2(r-1) = 4-p$. Since
$$
\gcd(4-p,p-1) = \gcd(p-1,3) = \gcd(q,3),
$$
the hypothesis $3\nmid q$ implies $\gcd(\ell,p-1)=1$, so Theorem~\ref{th: V_l count} applies.
The number of $f\in \mathcal F$ with $\deg f = r-1$ is exactly the number of sign vectors with $c_{r-1}=0$,
which is $\bigl|V_{r,\ell}\bigr|$ for this choice of $\ell$.
\end{proof}

\subsection{Heuristic and computations}

Let $f(x) = \sum_k c_k x^k$ be a polynomial that computes square roots. Motivated by the exact counting
result above and by extensive computations, we make the following heuristic assumptions.

\begin{conjecture}
\begin{enumerate}
\item[$(H_1)$] For each $k$,
$$
\mathbb{P}(c_k = 0) = \frac{1}{p} + O(2^{-r}).
$$

\item[$(H_2)$] For each divisor $d\mid r$, put $D = \frac{r}{d}$. Inside the family of polynomials whose
sign vectors have alternating order $d$, the events that the admissible coefficients vanish
(that is, the coefficients $c_k$ with $k \equiv r - \frac{D-1}{2} \pmod D$) are approximately
independent as $k$ ranges over this arithmetic progression.
\end{enumerate}
\end{conjecture}

\begin{remark}
\begin{enumerate}
\item As observed above, $(H_1)$ is proved (with an exact count) for all indices $k$ such that
$\gcd(1-2k,p-1)=1$ by Theorem~\ref{th: V_l count}. The second heuristic seems out of reach at present,
but it is strongly supported by numerical data.

\item Strictly speaking, $(H_2)$ is not true: the events $\{c_k=0\}$ along a given arithmetic progression
are not genuinely independent. Our working hypothesis is that the dependence is weak enough not to
affect the asymptotics of the expected minimal degree.

\item A simpler but numerically very similar model is to assume that the coefficients $c_k$ are
independent and uniformly distributed in $\F_p$. This would give the naive heuristic
$$
(H')\qquad
\mathbb{P}\bigl(c_i = 0 \text{ for all } i\in I\bigr) = p^{-|I|}
$$
for any finite index set $I$.
\end{enumerate}
\end{remark}

With these heuristics, we can estimate the expected number of square-root polynomials. We group
polynomials according to their alternating order $a(\epsilon)$ and set
$$
\mathcal F_d = \{f_\epsilon \in F : a(\epsilon) = d\},\qquad N(d) = \#\mathcal F_d.
$$
Every sign vector of length $d$ has alternating order dividing $d$, so
$$
\sum_{\ell\mid d} N(\ell) = 2^d.
$$
Hence
$$
N(d) = 2^d - \sum_{\substack{\ell\mid d\\ \ell<d}} N(\ell).
$$
Using the trivial bound $N(\ell)\le 2^\ell\le 2^{d/2}$ for $\ell\le d/2$, we obtain
$$
N(d) \ge 2^d - \tau(d)2^{d/2},
$$
where $\tau(d)$ is the number of positive divisors of $d$. In particular,
$$
\log N(d) = d\log 2 + O(2^{-d/2}).
$$
For large $d$ it is therefore reasonable to approximate
$$
N(d) \approx 2^d.
$$

Notice that when $d = 2^{s-1}$ (so that $D = \frac{r}{d} = q$), we are in the Tonelli–Shanks case,
and in fact $N(d) = 2^d$ exactly: every choice of the $2^{s-1}$ signs on the blocks yields a TS-polynomial.

We will need the following structural fact.

\begin{proposition}
Let $d\mid r$ and set $D = \frac{r}{d}$. Suppose $D$ is odd. Then the possible degrees of polynomials
$f_\epsilon \in \mathcal F_d$ are
$$
\deg(f_\epsilon) = r - \frac{D-1}{2} - nD,\qquad n = 0,1,2,\dots.
$$
\end{proposition}

\begin{proof}
This follows directly from Theorem~\ref{thm:orthogonality_criterion}, applied to sign vectors of
alternating order $d$: within such a family, all coefficients $c_k$ vanish except for
$k\equiv r - \frac{D-1}{2}\pmod D$, and the largest admissible index gives the degree.
\end{proof}

For example, when $D = q$ (the TS family), the possible degrees are
$$
\deg(f_{TS}) = r - \frac{q-1}{2} - n q,\qquad n\ge 0.
$$

\begin{remark}
If $f_\epsilon \in \mathcal F_d$, then there are $2d$ polynomials in the same orbit under the flip–shift map
$\sigma$ from Section~\ref{sec: Flip Shift}. These polynomials have the same degree and belong to the same
family $\mathcal F_d$. Thus, for counting purposes we may quotient by this action and only keep one
representative from each orbit. In other words, we may replace $N(d)$ by
$$
N'(d) = \frac{N(d)}{2d}
$$
as the relevant number of polynomials in $\mathcal F_d$.
\end{remark}

Under $(H_1)$–$(H_2)$, the expected number of polynomials in $\mathcal F_d$ whose degree is at most
$$
r - \frac{D-1}{2} - nD
$$
is
$$
\lambda_n^{(d)} = \frac{N'(d)}{p^{n}} \approx \frac{2^d}{2d\,p^n}.
$$
We look for the largest $n$ such that $\lambda_n^{(d)} \ge \tfrac12$. This condition is
$$
\frac{2^d}{2d\,p^{n}} \ge \frac12
\quad\Longleftrightarrow\quad
p^{n} \le \frac{2^d}{d}.
$$
Equivalently,
$$
n \le \frac{\log_2(2^d/d)}{\log_2 p}
= \frac{d - \log_2 d}{\log_2 p},
$$
so we expect
$$
n_{\max}^{(d)} \approx \left\lfloor \frac{d - \log_2 d}{\log_2 p} \right\rfloor.
$$
Thus the expected minimal degree within the family $\mathcal F_d$ is
$$
\min\deg(\mathcal F_d)
\approx r - \frac{D-1}{2} - D\left\lfloor \frac{d - \log_2 d}{\log_2 p} \right\rfloor.
$$

Globally,
$$
\min\{\deg f : f\in F\} = \min_{d\mid r} \min\deg(\mathcal F_d).
$$
This expression does not simplify to a closed form in $p$, but for each fixed prime it is straightforward
to evaluate numerically.

We now compare these predictions with numerical data. The algorithm and code used to compute the
minimal degrees are described in the appendix and can be found in \cite{Noah_Github}.
The “expected” counts in the tables are computed using the naive heuristic $(H')$.

\newpage 

\begin{longtable}{llllllll}
\caption{ts long}
\label{tab:ts-long}\\
\toprule
  p & s &  q & degree &        ALL\_Counts &      ALL\_expected & TS\_counts & TS\_expected \\
\midrule
\endfirsthead
\caption[]{ts long} \\
\toprule
  p & s &  q & degree &        ALL\_Counts &      ALL\_expected & TS\_counts & TS\_expected \\
\midrule
\endhead
\midrule
\multicolumn{8}{r}{{Continued on next page}} \\
\midrule
\endfoot

\bottomrule
\endlastfoot
 29 & 2 &  7 &     10 &               0 &                 1 &         0 &           0 \\
 29 & 2 &  7 &     11 &                 4 &                19 &         4 &           4 \\
 29 & 2 &  7 &     12 &               560 &               545 &         0 &           0 \\
 29 & 2 &  7 &     13 &             15820 &             15819 &         0 &           0 \\
 37 & 2 &  9 &     14 &                 4 &                 5 &         4 &           4 \\
 37 & 2 &  9 &     15 &               180 &               186 &         0 &           0 \\
 37 & 2 &  9 &     16 &              6984 &              6893 &         0 &           0 \\
 37 & 2 &  9 &     17 &            254976 &            255059 &         0 &           0 \\
 41 & 3 &  5 &     16 &               0 &                15 &         0 &           0 \\
 41 & 3 &  5 &     17 &               640 &               609 &         0 &           0 \\
 41 & 3 &  5 &     18 &             24936 &             24951 &        16 &          16 \\
 41 & 3 &  5 &     19 &           1023000 &           1023001 &         0 &           0 \\
 53 & 2 & 13 &     20 &                 4 &                 0 &         4 &           4 \\
 53 & 2 & 13 &     21 &                 0 &                 8 &         0 &           0 \\
 53 & 2 & 13 &     22 &               416 &               442 &         0 &           0 \\
 53 & 2 & 13 &     23 &             23660 &             23440 &         0 &           0 \\
 53 & 2 & 13 &     24 &           1242124 &           1242314 &         0 &           0 \\
 53 & 2 & 13 &     25 &          65842660 &          65842659 &         0 &           0 \\
 61 & 2 & 15 &     23 &                 4 &                 0 &         4 &           4 \\
 61 & 2 & 15 &     24 &                 0 &                 1 &         0 &           0 \\
 61 & 2 & 15 &     25 &               120 &                76 &         0 &           0 \\
 61 & 2 & 15 &     26 &              5100 &              4653 &         0 &           0 \\
 61 & 2 & 15 &     27 &            283500 &            283832 &         0 &           0 \\
 61 & 2 & 15 &     28 &          17292540 &          17313762 &         0 &           0 \\
 61 & 2 & 15 &     29 &        1056160560 &        1056139499 &         0 &           0 \\
 73 & 3 &  9 &     30 &               0 &                33 &         0 &           0 \\
 73 & 3 &  9 &     31 &              2088 &              2387 &         0 &           0 \\
 73 & 3 &  9 &     32 &            172480 &            174229 &        16 &          16 \\
 73 & 3 &  9 &     33 &          12716784 &          12718730 &         0 &           0 \\
 73 & 3 &  9 &     34 &         928438416 &         928467316 &         0 &           0 \\
 73 & 3 &  9 &     35 &       67778146968 &       67778114041 &         0 &           0 \\
 89 & 3 & 11 &     37 &               nan &                35 &         0 &           0 \\
 89 & 3 & 11 &     38 &              2904 &              3115 &         0 &           0 \\
 89 & 3 & 11 &     39 &            273784 &            277238 &        16 &          16 \\
 89 & 3 & 11 &     40 &          24706616 &          24674164 &         0 &           0 \\
 89 & 3 & 11 &     41 &        2195987728 &        2196000635 &         0 &           0 \\
 89 & 3 & 11 &     42 &      195444040704 &      195444056547 &         0 &           0 \\
 89 & 3 & 11 &     43 &    17394521032680 &    17394521032681 &         0 &           0 \\
 97 & 5 &  3 &     40 &               nan &                 3 &         0 &           0 \\
 97 & 5 &  3 &     41 &               608 &               334 &        32 &           7 \\
 97 & 5 &  3 &     42 &             33024 &             32440 &         0 &           0 \\
 97 & 5 &  3 &     43 &           3238272 &           3146678 &         0 &           0 \\
 97 & 5 &  3 &     44 &         305133760 &         305227801 &       640 &         669 \\
 97 & 5 &  3 &     45 &       29607090144 &       29607096698 &         0 &           0 \\
 97 & 5 &  3 &     46 &     2871888387456 &     2871888379660 &         0 &           0 \\
 97 & 5 &  3 &     47 &   278573172827392 &   278573172827041 &     64864 &       64860 \\
101 & 2 & 25 &     38 &               4 &                 0 &         4 &           4 \\
101 & 2 & 25 &     42 &               0 &                10 &         0 &           0 \\
101 & 2 & 25 &     43 &               600 &              1050 &         0 &           0 \\
101 & 2 & 25 &     44 &               115700 &            106065 &         0 &           0 \\
101 & 2 & 25 &     45 &               10713100 &          10712551 &         0 &           0 \\
101 & 2 & 25 &     46 &               1081898900 &        1081967680 &         0 &           0 \\
101 & 2 & 25 &     47 &               109278676200 &      109278735713 &         0 &           0 \\
101 & 2 & 25 &     48 &              11037152425620 &    11037152307054 &         0 &           0 \\
101 & 2 & 25 &     49 &               1114752383012500 &  1114752383012490 &         0 &           0 \\
109 & 2 & 27 &     41 &               4 &                 0 &         4 &           4 \\
109 & 2 & 27 &     45 &               0 &                 1 &         0 &           0 \\
109 & 2 & 27 &     46 &               0 &                98 &         0 &           0 \\
109 & 2 & 27 &     47 &               10728 &             10643 &         0 &           0 \\
109 & 2 & 27 &     48 &               1174176 &           1160071 &         0 &           0 \\
109 & 2 & 27 &     49 &               126427068 &         126447728 &         0 &           0 \\
109 & 2 & 27 &     50 &               13781571648 &       13782802392 &         0 &           0 \\
109 & 2 & 27 &     51 &               1502326744776 &     1502325460684 &         0 &           0 \\
109 & 2 & 27 &     52 &               163753475167764 &   163753475214549 &         0 &           0 \\
109 & 2 & 27 &     53 &               17849128798385820 & 17849128798385800 &         0 &           0 \\
113 & 4 &  7 &     46 &              0 &                 0 &         0 &           2 \\
113 & 4 &  7 &     47 &               0 &                 3 &         0 &           0 \\
113 & 4 &  7 &     48 &               672 &               304 &         0 &           0 \\
113 & 4 &  7 &     49 &             35168 &             34304 &         0 &           0 \\
113 & 4 &  7 &     50 &           3844064 &           3876387 &         0 &           0 \\
113 & 4 &  7 &     51 &         437844064 &         438031721 &         0 &           0 \\
113 & 4 &  7 &     52 &       49497868448 &       49497584522 &         0 &           0 \\
113 & 4 &  7 &     53 &     5593227289920 &     5593227051029 &       256 &         254 \\
113 & 4 &  7 &     54 &   632034656462160 &   632034656766225 &         0 &           0 \\
113 & 4 &  7 &     55 & 71419916214583400 & 71419916214583400 &         0 &           0 \\
\end{longtable}

\subsection{Tonelli-Shanks heuristics}

We end by specializing the above heuristics to the case of TS-polynomials.

Recall that the TS-polynomials correspond to the family $d=2^{s-1}$ and $D=q$. The discussion in the previous subsection gives us the following:

\begin{corollary}
    Let $p$ be an odd prime such that $p-1=2^s q$. Set $r=\frac{p-1}{2}, d=2^{s-1}, D=q$. Then, the TS-polynomials are of degrees $$r-\frac{q-1}{2}-nq$$ for $n\geq 0$. 
\end{corollary}
The expected number of TS-polynomials of that given degree is $$\lambda_n^{(d)}\approx \left (1-\frac 1p\right )\frac{N(d)}{p^n} = \left (1-\frac 1p\right )\frac{2^d}{p^n}$$

    The expected minimum degree of a TS-polynomial is $$\min \deg(f_{TS})= r - \frac{q-1}{2} - q \lfloor \frac{2^{s-1}-s+1}{\log_2 (p)} \rfloor$$

Computationally, this heuristic appears accurate for small primes $p$:
\newpage 

\begin{table}[H]
\centering
\small
\setlength{\tabcolsep}{6pt}
\begin{tabular}{r r r l l}
\toprule
$p$ & $s$ & $q$ &
\multicolumn{1}{c}{degree counts} &
\multicolumn{1}{c}{expected counts rounded} \\
\midrule
  7 & 1 &  3 & \texttt{\{ deg 2: 2 \}} & \texttt{\{ deg 2: 2 \}} \\
 11 & 1 &  5 & \texttt{\{ deg 3: 2 \}} & \texttt{\{ deg 3: 2 \}} \\
 13 & 2 &  3 & \texttt{\{ deg 5: 4 \}} & \texttt{\{ deg 5: 4 \}} \\
 17 & 4 &  1 & \texttt{\{ deg 7: 16, deg 8: 240 \}} & \texttt{\{ deg 6: 1, deg 7: 14, deg 8: 241 \}} \\
 19 & 1 &  9 & \texttt{\{ deg 5: 2 \}} & \texttt{\{ deg 5: 2 \}} \\
 23 & 1 & 11 & \texttt{\{ deg 6: 2 \}} & \texttt{\{ deg 6: 2 \}} \\
 29 & 2 &  7 & \texttt{\{ deg 11: 4 \}} & \texttt{\{ deg 11: 4 \}} \\
 31 & 1 & 15 & \texttt{\{ deg 8: 2 \}} & \texttt{\{ deg 8: 2 \}} \\
 37 & 2 &  9 & \texttt{\{ deg 14: 4 \}} & \texttt{\{ deg 14: 4 \}} \\
 41 & 3 &  5 & \texttt{\{ deg 18: 16 \}} & \texttt{\{ deg 18: 16 \}} \\
 43 & 1 & 21 & \texttt{\{ deg 11: 2 \}} & \texttt{\{ deg 11: 2 \}} \\
 47 & 1 & 23 & \texttt{\{ deg 12: 2 \}} & \texttt{\{ deg 12: 2 \}} \\
 53 & 2 & 13 & \texttt{\{ deg 20: 4 \}} & \texttt{\{ deg 20: 4 \}} \\
 59 & 1 & 29 & \texttt{\{ deg 15: 2 \}} & \texttt{\{ deg 15: 2 \}} \\
 61 & 2 & 15 & \texttt{\{ deg 23: 4 \}} & \texttt{\{ deg 23: 4 \}} \\
 67 & 1 & 33 & \texttt{\{ deg 17: 2 \}} & \texttt{\{ deg 17: 2 \}} \\
 71 & 1 & 35 & \texttt{\{ deg 18: 2 \}} & \texttt{\{ deg 18: 2 \}} \\
 73 & 3 &  9 & \texttt{\{ deg 32: 16 \}} & \texttt{\{ deg 32: 16 \}} \\
 79 & 1 & 39 & \texttt{\{ deg 20: 2 \}} & \texttt{\{ deg 20: 2 \}} \\
 83 & 1 & 41 & \texttt{\{ deg 21: 2 \}} & \texttt{\{ deg 21: 2 \}} \\
 89 & 3 & 11 & \texttt{\{ deg 39: 16 \}} & \texttt{\{ deg 39: 16 \}} \\
 97 & 5 &  3 & \texttt{\{ deg 41: 32, deg 44: 640, deg 47: 64864 \}} & \texttt{\{ deg 41: 7, deg 44: 669, deg 47: 64860 \}} \\
101 & 2 & 25 & \texttt{\{ deg 38: 4 \}} & \texttt{\{ deg 38: 4 \}} \\
103 & 1 & 51 & \texttt{\{ deg 26: 2 \}} & \texttt{\{ deg 26: 2 \}} \\
107 & 1 & 53 & \texttt{\{ deg 27: 2 \}} & \texttt{\{ deg 27: 2 \}} \\
109 & 2 & 27 & \texttt{\{ deg 41: 4 \}} & \texttt{\{ deg 41: 4 \}} \\
113 & 4 &  7 & \texttt{\{ deg 53: 256 \}} & \texttt{\{ deg 46: 2, deg 53: 254 \}} \\
127 & 1 & 63 & \texttt{\{ deg 32: 2 \}} & \texttt{\{ deg 32: 2 \}} \\
131 & 1 & 65 & \texttt{\{ deg 33: 2 \}} & \texttt{\{ deg 33: 2 \}} \\
137 & 3 & 17 & \texttt{\{ deg 60: 16 \}} & \texttt{\{ deg 60: 16 \}} \\
139 & 1 & 69 & \texttt{\{ deg 35: 2 \}} & \texttt{\{ deg 35: 2 \}} \\
149 & 2 & 37 & \texttt{\{ deg 56: 4 \}} & \texttt{\{ deg 56: 4 \}} \\
151 & 1 & 75 & \texttt{\{ deg 38: 2 \}} & \texttt{\{ deg 38: 2 \}} \\
157 & 2 & 39 & \texttt{\{ deg 59: 4 \}} & \texttt{\{ deg 59: 4 \}} \\
163 & 1 & 81 & \texttt{\{ deg 41: 2 \}} & \texttt{\{ deg 41: 2 \}} \\
167 & 1 & 83 & \texttt{\{ deg 42: 2 \}} & \texttt{\{ deg 42: 2 \}} \\
173 & 2 & 43 & \texttt{\{ deg 65: 4 \}} & \texttt{\{ deg 65: 4 \}} \\
179 & 1 & 89 & \texttt{\{ deg 45: 2 \}} & \texttt{\{ deg 45: 2 \}} \\
181 & 2 & 45 & \texttt{\{ deg 68: 4 \}} & \texttt{\{ deg 68: 4 \}} \\
191 & 1 & 95 & \texttt{\{ deg 48: 2 \}} & \texttt{\{ deg 48: 2 \}} \\
353 & 5 & 11 &  \texttt{\{ deg 160: 96, deg: 171: 65440  \}} & \texttt{\{ deg 149: 1, deg 160: 185, deg 171: 65536 \}}\\
\bottomrule
\end{tabular}
\caption{TS degree distributions and heuristic expected counts (rounded to nearest integer). For each $p$, $p-1=2^s q$ with $q$ odd.}
\end{table}

\newpage

\begin{appendices}

In this appendix, we give a more detailed discussion of the tree construction of root computing polynomials, and how it was used to build a fast algorithm for computing the minimal degree root computing polynomial for a given prime $p$.

\section{Tree Construction Of Root Computing Polynomials}

The following discussion is best motivated through an example. 
Assume $p\equiv 9 \mod 16$. Therefore, $s=3$ and one creates $S_0, S_1, S_2, S_3$ and $f_i \in F_i$. Instead of immediately gluing the $f_i$'s together using Equation \ref{eq: gluing Si}, to get an $f\in \mathcal F$, one can perform the process in steps, as in the following tree diagram:

$$
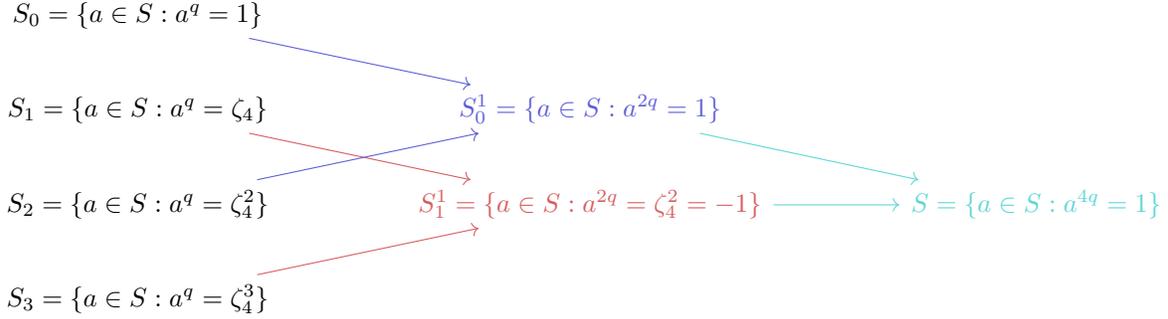
\begin{tikzcd}
	{S_0=\{a\in S: a^q = 1\}} \\
	{S_1=\{a\in S: a^q = \zeta_4\}} && \textcolor{rgb,255:red,92;green,92;blue,214}{{S_0^1=\{a \in S: a^{2q}=1\}}} \\
	{S_2=\{a\in S: a^q = \zeta_4^2\}} && \textcolor{rgb,255:red,214;green,92;blue,92}{{S_1^1=\{a \in S: a^{2q}=\zeta_4^2=-1\}}} && \textcolor{rgb,255:red,92;green,214;blue,214}{{S=\{a\in S: a^{4q}=1\}}} \\
	{S_3=\{a\in S: a^q = \zeta_4^3\}}
	\arrow[color={rgb,255:red,92;green,92;blue,214}, from=1-1, to=2-3]
	\arrow[color={rgb,255:red,214;green,92;blue,92}, from=2-1, to=3-3]
	\arrow[color={rgb,255:red,92;green,214;blue,214}, from=2-3, to=3-5]
	\arrow[color={rgb,255:red,92;green,92;blue,214}, from=3-1, to=2-3]
	\arrow[color={rgb,255:red,92;green,214;blue,214}, from=3-3, to=3-5]
	\arrow[color={rgb,255:red,214;green,92;blue,92}, from=4-1, to=3-3]
\end{tikzcd}$$
where $\zeta_4 \in \F_p$ is a primitive $4$-th root of unity.

For a general prime $p$, one can generalize this by setting 

$$S_i^k=\{a \in S_i: a^{2^kq}= \zeta_{2^{s-1}}^{2^ki}\}, \qquad i = 0 ,... ,2^{s-k-1}-1$$
and $S_i^0:= S_i$. Notice that as a variety, $S_i^k=Z(x^{2^kq}- \zeta_{2^{s-1}}^{2^ki})$. In more generality, given a subvariety $A \subseteq S$ with characterizing equation $g_A(x) = 0$, let $F_A$ be the set of polynomials in the coordinate ring $\faktor{\F_p[x]}{(g_A(x))}$ that compute square roots on $A$, i.e. $$F_A=\{f \in \faktor{\F_p[x]}{(g_A(x))}: f^2(x)=x\}$$

\begin{theorem} \label{th: F_A decomposition}
    Let $A$ be an algebraic subvariety of $S$ and $A_1 \cup A_2 \cup \ldots \cup A_n = A \subset S$, with $A_i$'s pairwise disjoint. Then, 

    \begin{equation*}
        F_{A} \cong F_{A_1} \times F_{A_2} \times \ldots \times F_{A_n}
    \end{equation*}
\end{theorem}

\begin{corollary} \label{cor: F_A decomposition}
    Write $\zeta=\zeta_{2^{s-1}}$. For $A=S_i^k$, denote $F_i^k=F_A$. Set $$M_k:=2^{k}q,\qquad \alpha:=\zeta^{2^{k}i}.$$ Then $$ F_i^k \times F_{i+2^{s-2-k}}^k \simeq F_i^{k+1}$$
    is given by 
    
$$
(f_i(x), f_j(x)) \mapsto f(x)=
\frac{1}{2\alpha}
\Bigl[
x^{M_k}\bigl(f_i(x)-f_j(x)\bigr)
+\alpha \bigl(f_i(x)+f_j(x)\bigr)
\Bigr]
$$

Consequently, 
\begin{equation}
\deg(f)= M^k +\deg(f_i-f_j)=2^kq + \deg(f_i-f_j) \label{eq:one level gluing}
\end{equation}
\end{corollary} 

\begin{proof}
    It follows immediately from the Chinese Remainder Theorem that

    \begin{equation*}
        \faktor{\F_p[x]}{(g_{A})} \cong \faktor{\F_p[x]}{(g_{A_1})} \times \faktor{\F_p[x]}{(g_{A_2})} \times \ldots \times \faktor{\F_p[x]}{(g_{A_n})}
    \end{equation*}

    It remains to show that this isomorphism preserves computing square roots. It is clear that if $f \in F_{A}$ then $f \pmod{g_{A_i}}$ computes square roots on $A_i$. Now, suppose $(f_1, \ldots f_2) \in F_{A_1} \times F_{A_2} \times \ldots \times F_{A_n}$, and let $f$ be their composition using the Chinese Remainder Theorem. Then, $\forall a \in A, \, \exists 1 \leq i \leq n : a \in A_i$, so writing $f(x) = h(x)g_{A_i}(x) + f_i(x)$, we see that $f(a) = h(a)g_{A_i}(a) + f_i(a) = 0 \pm\sqrt{a} = \pm\sqrt{a}$.
\end{proof}

\subsection{Lower Bounds on \texorpdfstring{$\deg f$}{deg f} for all odd primes \texorpdfstring{$p$}{p}}
We start by recalling a couple of known facts on the lower bound of the minimum degree of polynomials that compute square roots.

\begin{enumerate}
    \item When $p\equiv 3 \mod 4$, then $\min \deg(f) = \frac{p+1}{4}$
    \item When $p \equiv 1 \mod 4$, we have $\min \deg(f) \geq \frac{p-1}{3}$ by \cite[Theoroem ~1.1]{MR4774704}
\end{enumerate}
This second statement can be easily generalized to any $f\in F_i^k$ as follows:
\begin{theorem}
    Assume $s \geq 2$, $1 \leq k \leq s-1$ and let $f \in F_i^k$. Set $M_k=2^kq$. Then 
    $$   \frac{2M_k}{3}=\frac{p - 1}{3 \cdot 2^{s-k-1}}\leq \deg f \leq M_k=\frac{p-1}{2^{s-k}}$$

    For $k=0$ and $f \in F_i^0$, instead we have $$\deg(f) \geq \frac{q+1}{2}$$
\end{theorem}

\begin{proof}
The proof is almost identical to \cite[Theorem~1.1]{MR4774704}, simply adjusting it to make sure we are working over $F_i^k$. For convenience, we present it here again.

The upper bound is because $f(x)$ interpolates $M_k$ points. 

When $k=0$, this is Theorem \ref{th: minimal f_k that compute S_k}. So assume $k>0$.

Without loss of generality, assume $f(x) \in F_0^k$. Set $M_k=2^kq = \frac{p-1}{2^{s-k}}$. Recall that the vanishing polynomial on $S_0^k$ is $X^{M_k}-1$.

Assume by contradiction that $f(X)$ is of degree $d < \frac{2M_k}{3}$. Then, we have

    $$
    f^2(X)-X=(X^{M_k}-1)A(X)
    $$

    Rewrite this as \begin{equation}\label{eq:kedlaya1}
        f^2(X)=A(X)X^{M_k}+B(X), \qquad B(X)=X-A(X)
    \end{equation}
    
    We remark that:
    \begin{itemize}
        \item $\deg(f(X)) = d$
        \item $\deg(A(X)), \deg(B(X)) \leq 2d - M_k$
        \item $A(X) \neq 0$ as that would imply $f^2(X)=X$.
        \item $B(X) \neq 0$. Otherwise, $A(X) = X$, and $f^2(X) = X^{M_k + 1}$. But that would imply $M$ is odd and therefore $k=0$.
    \end{itemize}

    \bigskip

    Taking (formal) derivatives of both sides of \ref{eq:kedlaya1}, we get:

    \begin{equation}
        2f(X)f'(X) = A'(X) X^{M_k} + A(X) M_k X^{M_k-1} +B'(X) \label{eq:kedlaya2} 
    \end{equation}

    Computing $2f(X)^2f'(X)$ in two ways using \ref{eq:kedlaya1} and \ref{eq:kedlaya2}, we get:

    \begin{equation*}
   \underbrace{(2f'A - fA')X^{M_k} - fA M X^{M_k-1}}_{\equiv 0 \pmod{X^{M_k-1}}} + (2f'B - fB') = 0
    \end{equation*}

    Hence $$2f'(X)B(X)-f(X)B'(X) \equiv 0 \mod X^{M_k-1}$$

  By the degree bounds, 

$$\deg(2f'B), \deg(fB') \leq (d-1)+(2d-M_k) =3d-M_k-1 <M-1$$

so the congruence forces the polynomial identity $2f'B=fB'$

    Since $B(X) \neq 0$, we get $\frac{2f'(X)}{f(X)} = \frac{B'(X)}{B(X)}$. Since $\deg(f), deg(B) < p$, by a basic property of logarithmic derivatives, this implies $f(X)^2 = \lambda B(X)$ for some nonzero $\lambda$, contradicting the fact that $A(X) \neq 0$.
\end{proof}

We summarize our discussion so far in the following diagram: 
\begin{figure}[h]
\centering
\resizebox{\linewidth}{!}{%
\begin{tikzpicture}[
  >=Stealth,
  every node/.style={font=\small},
  lab/.style={text=purple!70!black},
  blueNode/.style={circle, draw=blue!60, fill=blue!20, inner sep=1.6pt},
  edge/.style={blue, line width=0.9pt}
]

\def\xA{0}
\def\xB{3.6}
\def\xC{7.2}
\def\xD{10.8}
\def\xE{14.4}

\node[lab] at (\xA,3.2) {$k=0$};
\node[lab] at (\xB,3.2) {$k=1$};
\node[lab] at (\xC,3.2) {$k$};
\node[lab] at (\xD,3.2) {$k=s-2$};
\node[lab] at (\xE,3.2) {$k=s-1$};

\node[lab] at (\xA,2.6) {$M_0=q$};
\node[lab] at (\xB,2.6) {$M_1=2q$};
\node[lab] at (\xC,2.6) {$M_k=2^kq$};
\node[lab] at (\xD,2.6) {$M_{s-2}=2^{s-2}q$};
\node[lab] at (\xE,2.6) {$M_{s-1}=2^{s-1}q=\dfrac{p-1}{2}$};

\node[blueNode] (A0)   at (\xA, 1.2) {};
\node[blueNode] (A1)   at (\xA, 0.0) {};
\node            at (\xA,-1.2) {$\vdots$};
\node[blueNode] (Aend) at (\xA,-2.4) {};

\node[blue,anchor=west] at ([xshift=2pt]A0.east) {$F_0^{0}$};
\node[blue,anchor=west] at ([xshift=2pt]A1.east) {$F_1^{0}$};
\node[blue,anchor=west] at ([xshift=2pt]Aend.east) {$F_{2^{\,s-1}}^{0}$};

\node[blueNode] (B0)   at (\xB, 0.8) {};
\node[blueNode] (B1)   at (\xB,-0.8) {};
\node            at (\xB,-2.0) {$\vdots$};
\node[blueNode] (Bend) at (\xB,-3.2) {};

\node[blue,anchor=west] at ([xshift=2pt]B0.east) {$F_0^{1}$};
\node[blue,anchor=west] at ([xshift=2pt]B1.east) {$F_1^{1}$};
\node[blue,anchor=west] at ([xshift=2pt]Bend.east) {$F_{2^{\,s-2}}^{1}$};

\node[blueNode] (C0)   at (\xC, 0.6) {};
\node            at (\xC,-0.6) {$\vdots$};
\node[blueNode] (Cend) at (\xC,-1.8) {};

\node[blue,anchor=west] at ([xshift=2pt]C0.east) {$F_0^{k}$};
\node[blue,anchor=west] at ([xshift=2pt]Cend.east) {$F_{2^{\,s-k-1}}^{k}$};

\node[blueNode] (D0) at (\xD, 0.4) {};
\node[blueNode] (D1) at (\xD,-0.4) {};

\node[blue,anchor=west] at ([xshift=2pt]D0.east) {$F_0^{\,s-2}$};
\node[blue,anchor=west] at ([xshift=2pt]D1.east) {$F_{1}^{\,s-2}$};

\node[blueNode] (E0) at (\xE, 0.0) {};
\node[blue,anchor=west] at ([xshift=2pt]E0.east) {$F_0^{\,s-1}=F_S$};

\node[lab] at (\xA,-4.2) {$\displaystyle \frac{q+1}{2}\ \le\ \deg(f)\ \le\ q-1$};
\node[lab] at (\xC,-4.2) {$\displaystyle \frac{2M_k}{3}\ \le\ \deg(f)\ \le\ M_k-1$};
\node[lab] at (\xE,-4.2) {$\displaystyle \frac{p-1}{3}\ \le\ \deg(f)\ \le\ \frac{p-1}{2}$};

\end{tikzpicture}
}
\end{figure}

In particular, if $f \in \mathcal F$ is minimal, then 
$$\frac{p-1}{3}=\frac{2M_{s-1}}{3} \leq \deg(f) \leq \deg(f_{TS})=M_{s-1}-\frac{q-1}{2} \leq M_{s-1}-1$$

\subsection{Algorithm for finding minimal root computing polynomials using tree construction}

We now restrict to primes $p\equiv 1\pmod 4$. A naïve strategy would enumerate all
$2^{\frac{p-1}{2}}$ root–computing polynomials on $S=\mu_{\frac{p-1}{2}}$ and then select those
of minimal degree. Since each polynomial can be represented modulo $X^{\frac{p-1}{2}}-1$ with
$O(p)$ coefficients, this costs
$$
O\!\left(p\cdot 2^{\frac{p-1}{2}}\right).
$$
Using the tree structure from \S4, we reduce this to
$$
O\!\left(p\cdot 2^{\frac{p-1}{4}}\right).
$$

Write $p-1=2^{s}q$ with $q$ odd (so $s\ge 2$). At level $k=s-2$ we have
$M_{s-2}=2^{\,s-2}q=\frac{p-1}{4}$ and a partition
$$
F^{\,s-1}_0 \;\leadsto\; F^{\,s-2}_0\ \sqcup\ F^{\,s-2}_1,
$$
with $|F^{\,s-2}_0|=|F^{\,s-2}_1|=2^{\frac{p-1}{4}}$. By Lemma \ref{lem: bijection of S_i and F_i} there
is a bijection $F^{\,s-2}_0 \cong F^{\,s-2}_1$.

\paragraph{Step 1 — Build the penultimate families.}
Compute all polynomials in $F^{\,s-2}_0$ (e.g.\ via the Lagrange/DFT formula of \S3.1), then obtain
$F^{\,s-2}_1$ from the bijection. Since each representative has degree $<M_{s-2}=O(p)$, the work
here is
$$
O\!\left(p\cdot 2^{\frac{p-1}{4}}\right).
$$

\paragraph{Step 2 — Order by leading–coefficient profiles.}
Represent each $f\in F^{\,s-2}_i$ by its coefficient vector modulo $X^{M_{s-2}}-1$,
$$
f(x)=\sum_{j=0}^{M_{s-2}-1} c_j\,x^j,
$$
and sort $F^{\,s-2}_0$ and $F^{\,s-2}_1$ lexicographically from highest degree to lowest (MSD
radix over $\mathbb{F}_p$). With $M_{s-2}=O(p)$ and $|F^{\,s-2}_i|=2^{\frac{p-1}{4}}$, this costs
$$
O\!\left(M_{s-2}\cdot 2^{\frac{p-1}{4}}\right)=O\!\left(p\cdot 2^{\frac{p-1}{4}}\right).
$$

\paragraph{Step 3 — Glue with maximal prefix match.}
For $f_0\in F^{\,s-2}_0$ and $f_1\in F^{\,s-2}_1$, the one–level gluing formula produces
$f\in F^{\,s-1}_0$ with
$$
\deg f \;=\; M_{s-2}\;+\;\deg(f_0-f_1)\qquad\text{(provided }f_0\neq f_1\text{)}.
$$
Thus minimizing $\deg f$ is equivalent to minimizing $\deg(f_0-f_1)$, i.e.\ maximizing the length
of the common high–degree prefix of the two coefficient vectors. After the MSD sort, perform a
bucketed two–pointer sweep: group entries with the same leading coefficient, then recurse to the next
coefficient within each bucket. Each pass is linear in the list sizes, so the total here remains
$$
O\!\left(p\cdot 2^{\frac{p-1}{4}}\right).
$$
(Edge case: the cancellation $f_0=f_1$ would give $\deg f<M_{s-2}=\frac{p-1}{4}$, which is ruled
out by the lower bound $\deg f\ge \tfrac{2}{3}M_{s-1}=\tfrac{p-1}{3}$ for $s\ge 2$; see
Theorem~4.2.)

\paragraph{Conclusion.}
Combining the three steps, we enumerate all minimal–degree root–computing polynomials in time
$$
O\!\left(p\cdot 2^{\frac{p-1}{4}}\right).
$$
(A parallel implementation simply shards the construction of $F^{\,s-2}_0$ and the bucketed
matching across threads; memory usage is linear in the number of polynomials stored.)

A multi-threaded C++ implementation of this can be found here \cite{Noah_Github}

\subsection{An example}

For $p=41$, a minimal polynomial in $\mathcal F$ is of degree $17$, while the Tonelli-Shanks polynomial is of degree $18$. An example of a minimal polynomial is 
 \begin{align*}
 f(x)=& 15x^{17} + 32x^{16} + 4x^{15} + 12x^{14} + 37x^{13} + 25x^{12} + 5x^{11} + x^{10} + x^{9}  \\ & + x^{8} + 9x^{6} + 29x^{5} + 6x^{4} + 36x^{3} + 7x^{2} + 33x + 35
 \end{align*} 
 while the Tonelli-Shanks polynomial is $$f_{TS}(x)= 26 x^{18} +10 x^{13} +11x^8 +36x^3$$

Below are the tree decompositions of these polynomials, as in Section 4.

First is the tree decomposition of $f_{TS}(x)$:

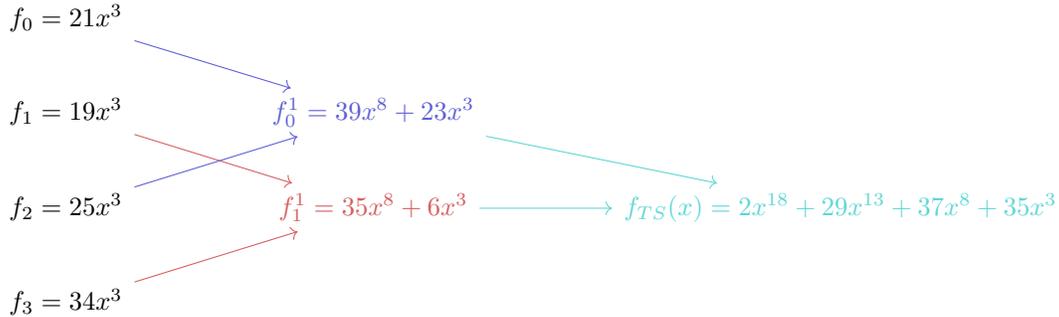
$$\begin{tikzcd}
	{f_0=21x^3} \\
	{f_1=19x^3} && \textcolor{rgb,255:red,92;green,92;blue,214}{{f_0^1=39x^8+23x^3}} \\
	{f_2=25x^3} && \textcolor{rgb,255:red,214;green,92;blue,92}{{f_1^1= 35x^8 +6x^3}} && \textcolor{rgb,255:red,92;green,214;blue,214}{{f_{TS}(x)=2x^{18}+29x^{13}+37x^8+35x^3}} \\
	{f_3=34x^3}
	\arrow[color={rgb,255:red,92;green,92;blue,214}, from=1-1, to=2-3]
	\arrow[color={rgb,255:red,214;green,92;blue,92}, from=2-1, to=3-3]
	\arrow[color={rgb,255:red,92;green,214;blue,214}, from=2-3, to=3-5]
	\arrow[color={rgb,255:red,92;green,92;blue,214}, from=3-1, to=2-3]
	\arrow[color={rgb,255:red,92;green,214;blue,214}, from=3-3, to=3-5]
	\arrow[color={rgb,255:red,214;green,92;blue,92}, from=4-1, to=3-3]
\end{tikzcd}$$

And second is the tree decomposition of $f(x)$:

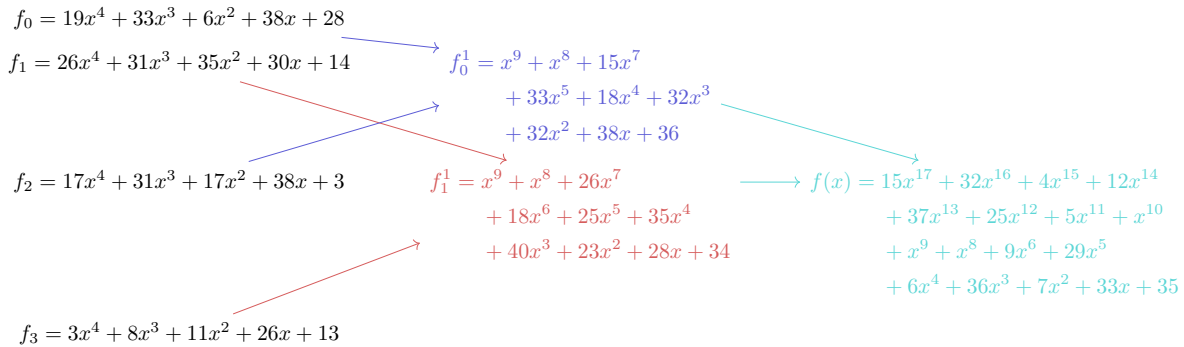
$$
\scalebox{0.82}{%
\begin{tikzcd}[column sep=1.4em, row sep=0.6ex, ampersand replacement=\&]
  {f_0=19x^4+33x^3+6x^2+38x+28} \\
  {f_1=26x^4+31x^3+35x^2+30x+14}
  \&\&
  \textcolor[RGB]{92,92,214}{%
    \begin{aligned}[t]
      f_0^{1}=\;& x^{9}+x^{8}+15x^{7}\\
               & {}+33x^{5}+18x^{4}+32x^{3}\\
               & {}+32x^{2}+38x+36
    \end{aligned}}
  \\
  {f_2=17x^4+31x^3+17x^2+38x+3}
  \&\&
  \textcolor[RGB]{214,92,92}{%
    \begin{aligned}[t]
      f_1^{1}=\;& x^{9}+x^{8}+26x^{7}\\
               & {}+18x^{6}+25x^{5}+35x^{4}\\
               & {}+40x^{3}+23x^{2}+28x+34
    \end{aligned}}
  \&\&
  \textcolor[RGB]{92,214,214}{%
    \begin{aligned}[t]
      f(x)=\;& 15x^{17}+32x^{16}+4x^{15}+12x^{14}\\
            & {}+37x^{13}+25x^{12}+5x^{11}+x^{10}\\
            & {}+x^{9}+x^{8}+9x^{6}+29x^{5}\\
            & {}+6x^{4}+36x^{3}+7x^{2}+33x+35
    \end{aligned}}
  \\
  {f_3=3x^4+8x^3+11x^2+26x+13}
  \arrow[color={rgb,255:red,92;green,92;blue,214}, from=1-1, to=2-3]
  \arrow[color={rgb,255:red,214;green,92;blue,92}, from=2-1, to=3-3]
  \arrow[color={rgb,255:red,92;green,214;blue,214}, from=2-3, to=3-5]
  \arrow[color={rgb,255:red,92;green,92;blue,214}, from=3-1, to=2-3]
  \arrow[color={rgb,255:red,92;green,214;blue,214}, from=3-3, to=3-5]
  \arrow[color={rgb,255:red,214;green,92;blue,92}, from=4-1, to=3-3]
\end{tikzcd}
}%
$$

We observe that for the Tonelli-Shanks polynomial, there is no degree reduction at any level of the tree, while for the minimal polynomial $f(x)$, there is a degree reduction in the last step. This pattern seems to hold for every polynomial which we have tested: If $f(x)$ is a \textit{minimal} polynomial, different from the Tonelli-Shanks, then the degree cancellation happens only at the last step of the tree. We leave this as an open question for future research.
\end{appendices}

\newpage
\printbibliography
\end{document}